\documentclass[12pt]{amsart}
\usepackage{xspace}
\usepackage[all]{xypic}
\usepackage{amsmath,amstext,amsfonts}
\usepackage[mathscr]{euscript}
\usepackage{amscd}
\usepackage{latexsym}
\usepackage{amssymb}
\theoremstyle{plain}
\swapnumbers
    \newtheorem{thm}{Theorem}[section]
    \newtheorem{prop}[thm]{Proposition}
    \newtheorem{lemma}[thm]{Lemma}
    \newtheorem{cor}[thm]{Corollary}
    
    \newtheorem{subsec}[thm]{}
        \newtheorem*{propa}{Proposition}
\theoremstyle{definition}
    
    \newtheorem{defn}[thm]{Definition}
    \newtheorem{example}[thm]{Example}

\theoremstyle{remark}
        \newtheorem{remark}[thm]{Remark}
        
    \newtheorem{ack}[thm]{Acknowledgements}
\newenvironment{myeq}[1][]
{\stepcounter{thm}\begin{equation}\tag{\thethm}{#1}}
{\end{equation}}

\newcommand{\mydiagram}[2][]
{\stepcounter{thm}\begin{equation}
     \tag{\thethm}{#1}\vcenter{\xymatrix{#2}}\end{equation}}
\newenvironment{mysubsection}[2][]
{\begin{subsec}\begin{upshape}\begin{bfseries}{#2.}
\end{bfseries}{#1}}
{\end{upshape}\end{subsec}}
\newenvironment{mysubsect}[2][]
{\begin{subsec}\begin{upshape}\begin{bfseries}{#2\vsn.}
\end{bfseries}{#1}}
{\end{upshape}\end{subsec}}
\newcommand{\wh}{\ -- \ }
\newcommand{\w}[2][ ]{\ \ensuremath{#2}{#1}\ }
\newcommand{\ww}[1]{\ \ensuremath{#1}}
\newcommand{\wb}[2][ ]{\ (\ensuremath{#2}){#1}\ }
\newcommand{\wref}[2][ ]{\ \eqref{#2}{#1}\ }
\newcommand{\hsp}{\hspace*{7 mm}}

\newcommand{\hsm}{\hspace{2 mm}}
\newcommand{\vsn}{\vspace{1 mm}}

\newcommand{\vsm}{\vspace{3 mm}}
\newcommand{\xra}[1]{\xrightarrow{#1}}
\newcommand{\xRa}[1]{\stackrel{#1}{\Longrightarrow}}
\newcommand{\Ra}{\Rightarrow}

\newcommand{\hra}{\hookrightarrow}
\newcommand{\bllt}{\raisebox{0.6ex}{\hbox{\circle*{3}}}}
\newcommand{\rest}[1]{\lvert_{#1}}
\newcommand{\ssr}{\!\!\!\!\!\!\!\!\!\!\!\!\!\!}
\newcommand{\lra}[1]{\langle{#1}\rangle}
\newcommand{\EQUIV}{\Leftrightarrow}

\newcommand{\adj}[2]{\substack{{#1}\\ \rightleftharpoons \\ {#2}}}

\newcommand{\ocmp}[1]{{#1}^{\otimes}}
\newcommand{\tens}[2]{#1\,\tiund{#2}\,#1}
\newcommand{\toto}{\substack{{\to}\\ \to}}
\newcommand{\dm}{\partial_{\max}}
\newcommand{\ovl}[1]{\overline{#1}}

\newcommand{\bsim}{/\!\!\sim}
\newcommand{\ch}{\circ_{h}}
\newcommand{\cv}{\boxplus}

\newcommand{\tiund}[1]{{\times}_{#1}}

\newcommand{\ab}{\operatorname{ab}}
\newcommand{\Arr}{\operatorname{Arr}}
\newcommand{\BW}{\operatorname{BW}}
\newcommand{\cf}{\operatorname{cf}}
\newcommand{\colim}{\operatorname{colim}}
\newcommand{\cub}{\operatorname{cub}}
\newcommand{\Dec}{\operatorname{Dec}\,}
\newcommand{\diag}{\operatorname{diag}}
\newcommand{\Fac}{\operatorname{Fac}}

\newcommand{\ho}{\operatorname{ho}}
\newcommand{\hocolim}{\operatorname{hocolim}}

\newcommand{\Hom}{\operatorname{Hom}}
\newcommand{\Id}{\operatorname{Id}}
\newcommand{\Ker}{\operatorname{Ker}}

\newcommand{\Obj}{\operatorname{Obj}\,}
\newcommand{\op}{^{\operatorname{op}}}
\newcommand{\orr}{\operatorname{or}}
\newcommand{\os}{\orr^{\ast}}
\newcommand{\res}{_{\operatorname{res}}}
\newcommand{\Str}{\operatorname{st}}
\newcommand{\csk}[1]{\operatorname{csk}_{#1}}

\newcommand{\dsk}[1]{\operatorname{dsk}_{#1}}
\newcommand{\sk}[1]{\operatorname{sk}_{#1}}
\newcommand{\skc}[1]{\sk{#1}^{c}}
\newcommand{\skh}[1]{\operatorname{sk}^{h}_{#1}}
\newcommand{\skv}[1]{\operatorname{sk}^{v}_{#1}}
\newcommand{\supar}[1]{\overset{#1}{-\!\!-\!\!\!\rightarrow}}

\newcommand{\map}{\operatorname{map}}

\newcommand{\mape}[2]{\map\sp{#1}\sb{#2}}

\newcommand{\Po}[1]{P\sp{#1}}
\newcommand{\C}{{\mathcal C}}

\newcommand{\D}{{\mathcal D}}
\newcommand{\E}{{\mathcal E}}
\newcommand{\F}{{\mathcal F}}
\newcommand{\Fs}{\F_{s}}
\newcommand{\G}{{\mathcal G}}
\newcommand{\K}{{\mathcal K}}
\newcommand{\M}{{\mathcal M}}
\newcommand{\N}{{\mathcal N}}
\newcommand{\tN}{\tilde{\N}}
\newcommand{\dN}{\N_{d}}
\newcommand{\eN}{\widehat{\dN}}

\newcommand{\OO}{{\mathcal O}}
\newcommand{\PP}{{\mathcal Q}}
\newcommand{\cP}{{\mathcal P}}
\newcommand{\dP}{\PP_{t}}
\newcommand{\Ph}{\PP\sp{h}}
\newcommand{\hP}{\widehat{\PP}}
\newcommand{\Ss}{{\mathcal S}}
\newcommand{\PS}[1]{\Po{#1}\Ss}
\newcommand{\St}{\PS{2}}
\newcommand{\Scf}{\Ss_{\cf}}
\newcommand{\T}{{\EuScript Top}}

\newcommand{\Tam}{{\EuScript Tam_{2}}}
\newcommand{\V}{{\mathcal V}}
\newcommand{\td}{\tilde{d}}
\newcommand{\tphi}{\tilde{\phi}}
\newcommand{\tpsi}{\tilde{\psi}}

\newcommand{\hy}[2]{{#1}\text{-}{#2}}
\newcommand{\Ab}{{\EuScript Ab}}
\newcommand{\Abgp}{{\Ab\Gp}}
\newcommand{\Aug}{{\EuScript Aug}\,}
\newcommand{\Cat}{{\EuScript Cat}}
\newcommand{\cat}{\operatorname{cat}}

\newcommand{\DG}{{\EuScript D}i{\EuScript G}}
\newcommand{\Gp}{{\EuScript Gp}}
\newcommand{\Gpd}{{\EuScript Gpd}}
\newcommand{\TGpd}{{\EuScript DbGpd}}
\newcommand{\TGw}{\TGpd_{t}}
\newcommand{\BGpd}{{\EuScript BiGpd}}
\newcommand{\ud}{\sp{\delta}}
\newcommand{\Gd}{\Gpd\ud}

\newcommand{\NS}{{\EuScript NS}}
\newcommand{\Set}{{\EuScript Set}}

\newcommand{\VC}{\hy{\V}{\Cat}}
\newcommand{\Sb}{S_{\cub}\,}

\newcommand{\CO}{(\C,\OO)}
\newcommand{\COC}{\hy{\CO}{\Cat}}
\newcommand{\OC}{\hy{\OO}{\Cat}}
\newcommand{\GD}{(\Gpd,\D_{0})}
\newcommand{\GDC}{\hy{\GD}{\Cat}}
\newcommand{\GdD}{(\Gd,\D_{0})}
\newcommand{\GdDC}{\hy{\GdD}{\Cat}}
\newcommand{\GE}{(\Gpd,\E)}
\newcommand{\GO}{(\Gpd,\OO)}
\newcommand{\GOC}{\hy{\GO}{\Cat}}
\newcommand{\TraO}{\GOC}
\newcommand{\TTra}{\hy{2}{\mathrm{Track}}}

\newcommand{\TGwO}{(\TGw,\OO)}
\newcommand{\TGwOC}{\hy{\TGwO}{\Cat}}
\newcommand{\TTraO}{\TGwOC}
\newcommand{\SO}{(\Ss,\OO)}
\newcommand{\SOC}{\hy{\SO}{\Cat}}

\newcommand{\PCO}[1]{(\Po{#1}\C,\OO)}

\newcommand{\PSO}[1]{(\PS{#1},\OO)}
\newcommand{\PSOC}[1]{\hy{\PSO{#1}}{\Cat}}
\newcommand{\VO}{(\V,\OO)}
\newcommand{\VOC}{\hy{\VO}{\Cat}}
\newcommand{\Ic}[1]{{\mathcal I}^{#1}}
\newcommand{\EM}[3]{E_{#1}({#2},{#3})}
\newcommand{\ba}{\vec{\mathbf{a}}}
\newcommand{\bDelta}{\mathbf{\Delta}}
\newcommand{\Fp}{{\mathbb F}_{p}}

\newcommand{\bull}{\,\hbox{\circle*{3}}}
\newcommand{\ubull}{\,\,\raisebox{1.2ex}{\hbox{\circle*{3}}}}
\newcommand{\Bd}{B\bull}
\newcommand{\Cu}{C\ubull}
\newcommand{\Cus}{C^{\ast}}
\newcommand{\Ed}{E\bull}

\newcommand{\fd}{f\bull}
\newcommand{\gd}{g\bull}

\newcommand{\Gss}{G_{\ast\ast}}
\newcommand{\hd}{h\bull}
\newcommand{\Md}{M\bull}
\newcommand{\Nd}{N\bull}
\newcommand{\tNd}{\tN\bull}

\newcommand{\tU}{\tilde{U}}
\newcommand{\tUd}{\tU\bull}
\newcommand{\Vd}{V\bull}
\newcommand{\tVd}{\tilde{V}\bull}
\newcommand{\Vdd}{V\bull\bull}

\newcommand{\Wd}{W\bull}
\newcommand{\Wdd}{W\bull\bull}
\newcommand{\tW}{\tilde{W}}
\newcommand{\tWd}{\tW\bull}

\newcommand{\Xd}{X\bull}
\newcommand{\tX}{\tilde{X}}
\newcommand{\tXd}{\tX\bull}
\newcommand{\Yd}{Y_{\bullet}}
\newcommand{\tY}{\tilde{Y}}
\newcommand{\tYd}{\tY\bull}
\newcommand{\Zd}{Z\bull}
\newcommand{\tZ}{\tilde{Z}}
\newcommand{\tZd}{\tZ\bull}
\newcommand{\Xdd}{X\bull\bull}

\newcommand{\ui}[1]{^{({#1})}}
\newcommand{\hPi}{\hat{\Pi}_{1}}

\newcommand{\hpi}{\hat{\pi}_{1}}

\newcommand{\bk}{[\mathbf{k}]}
\newcommand{\bm}{[\mathbf{m}]}
\newcommand{\bn}{[\mathbf{n}]}

\newcommand{\bze}{[\mathbf{0}]}

\begin{document}
\title{Two-track categories}
\author{David Blanc and Simona Paoli}
\address{Department of Mathematics\\ University of Haifa\\ 31905 Haifa\\ Israel}
\email{blanc@math.haifa.ac.il}
\email{paoli@math.haifa.ac.il}
\date{July 6, 2009. Corrected: December 17, 2009}
\subjclass{Primary: 18G55; \ secondary: 18B40,18G30,55S45,55N99}
\keywords{Track categories, double groupoids, $2$-types, simplicially enriched
  categories, Baues-Wirsching cohomology}

\begin{abstract}
We describe a $2$-dimensional analogue of track categories, called
two-track categories, and show that it can be used to model categories
enriched in $2$-type mapping spaces. We also define a Baues-Wirsching
type cohomology theory for track categories, and explain how it can be
used to classify two-track extensions of a track category $\D$ by a
module over $\D$.
\end{abstract}

\maketitle

\setcounter{section}{0}

%
%
\section*{Introduction}
\label{cintr}

In \cite{DKanF}, Dwyer and Kan showed that any model category $\E$ can be
provided with simplicial function complexes, in such a way that the resulting
simplicially enriched category \w{\Xd} encodes the homotopy theory of $\E$.
As in the case of individual topological spaces, it is often
useful to approximate \w{\Xd} by its Postnikov sections \w{\Po{n}\Xd}
for \w[,]{n\geq 0} obtained by applying the $n$-th Postnikov functor
to each mapping space of \w[.]{\Xd} Hence \w{\Po{n}\Xd} is a category
``enriched in $n$-types'' \wh that is, in simplicial sets whose
homotopy groups vanish in dimensions \w[.]{>n} Consecutive Postnikov
sections are related as usual via their $k$-invariants, which take
value in certain \ww{\SO}-cohomology groups (see \cite{DKanS}).

It is convenient to have algebraic models for \w[,]{\Po{n}\Xd} in
particular if they allow for an explicit description of the (systems
of) homotopy groups (see \S \ref{egnatmod}), Postnikov towers, and
$k$-invariants for \w[.]{\Xd} 

For instance, in the case \w{n=1} the fundamental groupoid \w{\hpi Y}
of an individual simplicial set (or topological space) $Y$ provides an
algebraic model the \emph{$1$-type} of $Y$ \wh that is, the homotopy
type of \w[.]{\Po{1}Y} If we use Kan's version of the Postnikov
section, so that both functors \w{\hpi} and \w{\Po{1}} strictly
commute with products, they extend to simplicially enriched
categories.  Moreover, the nerve functor from groupoids to simplicial
sets lands in $1$-types, and is also monoidal, so it extends to the
enriched setting, providing an inverse up to homotopy to \w[.]{\hpi}
As a result, categories enriched in groupoids, called
\emph{track categories}, provide an algebraic model for
\w[,]{\Po{1}\Xd} up to homotopy. 

Ideally, such an algebraic description would allow one to better
understand the homotopy theory of the original model category $\E$ as
a whole \wh e.g., by providing an explicit calculus of higher homotopy
(or cohomology) operations.
For example, let \w{\E_{Y}} consist of a single space $Y$ together
with all maps from $Y$ into finite products of mod-$p$ Eilenberg-Mac~Lane
spaces, and maps between them. Its homotopy category \w{\ho\E_{Y}} thus
encodes the mod $p$ cohomology of $Y$ as an algebra over the Steenrod algebra.
If \w{\Xd} is the corresponding simplicially enriched category, its
track category \w{\hpi\Xd} records in addition all \emph{secondary}
mod-$p$ cohomology operations on $Y$. Moreover, in this case
\w{\hpi\Xd} can be described as a ``linear track extension'' of
\w{\ho\E_{Y}}  by a certain natural system $\K$ on \w{\ho\C} (see
\S \ref{ccstc} below), and this extension is classified by a class in
the third Baues-Wirsching cohomology group
\w[.]{H^{3}_{\BW}(\ho\C;\E_{Y})} This class may be identified in turn
with the $0$-th $k$-invariant for \w{\Xd} (see \cite[Theorem 6.5]{BBlaC}).
Baues and his collaborators have shown how this description can be used in
practice to elucidate the secondary structure of \w{H^{\ast}(Y;\Fp)}
and make computations in the Adams spectral sequence (see
\cite{BauAS,BJiblSD}).

\begin{mysubsection}{The $2$-dimensional case}\label{stdc}
In this paper we describe an approach to the $2$-dimensional case, which
both provides a setting for studying the ``tertiary Steenrod
algebra'' in the spirit of Baues's work, and indicates how one might
be able to proceed to higher dimensions.

This approach involves two main steps:
\begin{enumerate}
\renewcommand{\labelenumi}{(\alph{enumi})\ }
\item We first construct a functorial model of $2$-types, which we call
\emph{two-typical} double groupoids.

Many algebraic structures have been shown to model $2$-types of
topological spaces, beginning (in the connected case) with the crossed
modules of \cite{MWhitT} and the double groupoids with connections of
\cite{BSpenD}.  More general models include the homotopy double
groupoids of \cite{BHKPortH}, the homotopy bigroupoids of
\cite{HKKieHB}, the strict $2$-groupoids of \cite{MSvenA}, and the
weak $2$-groupoids of \cite{TamsN}.

The advantage of the two-typical double groupoids is the explicit
description of the model associated to a Kan complex, its homotopy
groups, and its Postnikov tower. 
\item More directly relevant to our present purpose is the fact that
  the resulting functor \w{\dP:\Ss\to\TGw} preserves products. Thus by
  applying $\dP$ to each mapping space of a simplicially enriched
  category \w[,]{\Xd} we obtain a convenient model for ``categories
  enriched in $2$-types'' \wh more precisely, for
  \emph{\ww{\PSO{2}}-categories}, which are categories \w{\Yd}
  enriched in simplicial spaces for which
  \w{\map_{\Yd}(a,b)\simeq\Po{2}\map_{\Yd}(a,b)} for each 
\w[.]{a,b\in\OO=\Obj\Yd}
\end{enumerate}

In the second part, we define a Baues-Wirsching type
cohomology \w{H^{\ast}_{\BW}(\D;\M)} for an (ordinary) track category $\D$,
with coefficients in an appropriate notion of a natural
system $\M$ on $\D$. We then show how one can associate an ``underlying
homotopy track category'' \w{\D=\ho\G} to any two-track category $\G$,
as well as a natural system \w{\M=\Pi_{2}\G} on $\D$, and explain how $\G$,
thought of as an extension of $\D$ by $\M$, is classified by a
naturally defined class \w[.]{\chi_{\G}\in H^{4}_{\BW}(\D;\M)}
Finally, we show that this cohomology theory is naturally isomorphic
(with a shift in dimension) to the \ww{\SO}-cohomology of Dwyer and
Kan for the corresponding simplicially enriched category \w[,]{\Xd=\N\D}
and that \w{\chi_{\G}} corresponds to the first $k$-invariant of
\w[.]{\Xd}

As with most new constructions of a known cohomology theory, one
should view the cohomology theory \w{H^{\ast}_{\BW}}
mainly as an alternative approach to the \emph{computation} of the
\ww{\SO}-cohomology of a \ww{\PSO{2}}-category. Since it is very
difficult to use the original definition of Dwyer and Kan to carry out
explicit calculations, it is to be hoped that our definition will make
it more accessible \wh in particular, to Baues-Jibladze type
computations of the Adams spectral sequence (see \cite{BJiblE,BJiblSD} and
\cite{BBlaS}).

See \cite{BauHOT} for a different approach to the $2$ (and higher
$n$)-dimensional cases.
\end{mysubsection}

\begin{mysubsection}{Notation and conventions}\label{snot}
For any category $\C$, \w{s\C:=\C^{\bDelta\op}} is the category of simplicial
objects over $\C$. We abbreviate \w{s\Set} to \w[.]{\Ss} If $\C$ is
concrete, the \emph{$n$-skeleton} \w{\sk{n}\Xd\in s\C} of any \w{\Xd\in s\C}
is generated under the degeneracy maps by \w[.]{X_{0},\dotsc, X_{n}}
The \emph{$n$-coskeleton} functor \w{\csk{n}:s\C\to s\C} is left
adjoint to \w[.]{\sk{n}}

We denote by \w{\Scf} the full subcategory of $\Ss$ consisting of Kan
complexes \wh i.e., fibrant (and cofibrant) simplicial sets.
For \w[,]{X\in\Scf} we can use \w{\csk{n+1}X} as a model for the $n$-th
Postnikov section \w[.]{\Po{n}X}
For each \w[,]{n\geq 0} let \w{\PS{n}} denote the full
subcategory of $\Ss$ consisting of simplicial sets $X$ for which the
natural map \w{X\to \Po{n}X} is a weak equivalence (that is,
\w{\pi_{i}(X,x)=0} for all \w{x\in X} and \w[).]{i>n} An
\emph{$n$-type} is (the homotopy equivalence class of) an object in
\w[.]{\PS{n}}

For a bisimplicial set \w[,]{\Wdd\in s\Ss=ss\Set} we think of the
first index as the \emph{horizontal} direction and the second index as
the \emph{vertical} direction, and for each (fixed) \w{n\geq 0} we
write \w{W_{n}^{h}\in\Ss} for the simplicial set with
\w{(W_{n}^{h})_{i}:=W_{n,i}} for \w[.]{i\geq 0}
Similarly, if \w{f:\Wdd\to\Vdd} is a map in \w[,]{s\Ss} we write \w{f_{n}^{h}}
for its restriction to \w[.]{W_{n}^{h}\in\Ss}

For any \w[,]{n\geq 0} a map \w{f:\Xd\to\Yd} in $\Ss$ is called an
\emph{$n$-equivalence} if it induces isomorphisms
\w{f_{\ast}:\pi_{0}\Xd\to\pi_{0}\Yd} (of sets), and
\w{f_{\#}:\pi_{i}(\Xd,x)\to\pi_{i}(\Yd,f(x))} for every \w{1\leq i\leq n} and
\w[.]{x\in X_{0}}

Let \w{\Cat} denote the category of small categories, and \w{\Gpd}
the full subcategory of groupoids.
If \w{\lra{\V,\otimes}} is a monoidal category, we denote by \w{\VC}
the collection of all (not necessarily small) categories enriched over $\V$
(see \cite[\S 6.2]{BorcH2}). For any set $\OO$, denote by \w{\OC} the
category of all small categories $\D$ having \w[,]{\Obj\D=\OO} with
functors which are the identity on objects as morphisms. A
\ww{\VO}-\emph{category} is a category \w{\D\in\OC} enriched over
$\V$, with mapping objects \w[.]{\mape{v}{\D}(-,-)\in\V}
The category of all \ww{\VO}-categories will be denoted by \w[.]{\VOC}

The main examples of \w{\lra{\V,\otimes}} to keep in mind are
\w[,]{\lra{\Set,\times}} \w[,]{\lra{\Gp,\times}} \w[,]{\lra{\Gpd,\times}}
\w[,]{\lra{\Ss,\times}} and \w[,]{\lra{\C,\otimes}} where $\C$ is the
category of cubical sets (see \S \ref{dcube} below).

We obtain further variants by applying any (strictly) monoidal functor
\w{P:\lra{\V,\otimes}\to\lra{\V',\otimes'}} to a \ww{\VO}-category
$\C$, where by a slight abuse of notation we use the same name for the
prolonged functor. For example, given an \ww{\SO}-category \w[,]{\Xd} for each
\w{n\geq 1} we have a \ww{\PSO{n}}-category \w[,]{\Yd:=\Po{n}\Xd} in which each
mapping space \w{\Yd(a,b)} is the $n$-th Postnikov section \w[.]{\Po{n}\Xd(a,b)}
More precisely, we use a functorial product-preserving fibrant
replacement for each mapping space of \w{\Xd} to obtain a fibrant
\ww{\SO}-category \w[,]{\tXd} and then apply a product-preserving
version \w{\Po{n}} of the $n$-th Postnikov section functor (cf.\
\cite[VI, \S 2]{GJarS}) to each mapping space of \w{\tXd} to obtain \w[.]{\Yd}

Note that because the Cartesian product on $\Ss$ is defined levelwise,
we can think of an \ww{\SO}-category as a simplicial object in
\w[,]{\OC} thus identifying \w{\SOC} with \w[.]{s\OC}

Finally, for any category $\C$, the category of abelian group objects
in $\C$ is denoted by \w[.]{\C_{\ab}}
\end{mysubsection}

\begin{mysubsection}{Organization}\label{sorg}
Section \ref{ctrack} provides a review of groupoids and track
categories. Section \ref{cdgpd} discusses double groupoids, and in
particular those which are \emph{two-typical} (Definition
\ref{dtwotypical}). In Section \ref{cdgttype} we show that two-typical
double groupoids model $2$-types (Theorem \ref{tmodeltt}), and in
Section \ref{cttrack} we use this notion to define two-track
categories, and show their equivalence with \ww{\PSO{2}}-categories,
up to weak equivalence (Corollary \ref{cmodeltt}). An alternative
model using Gray categories is given by Proposition \ref{pgray}.

Coefficient systems on track categories are defined in Section
\ref{ccstc}, and are used in Section \ref{cctc} to define the
cohomology of track categories (extending the Baues-Wirsching
cohomology of small categories), and show its equivalence
with the \ww{\SO}-cohomology of Dwyer and Kan (Theorem \ref{tbwso}).
Finally, in Section \ref{cttetc} we discuss two-track extensions of
track categories and show how they are classified by a suitable
Baues-Wirsching type cohomology class, which may be identified with
the first $k$-invariant of the corresponding \ww{\PSO{2}}-category
(Theorem \ref{tbwclass}).
\end{mysubsection}

\begin{ack}
We would like to thank Hans Baues and the referee for many useful
comments and corrections. This research was supported by BSF grant 2006039.
\end{ack}

%
%
\section{Groupoids and track categories}
\label{ctrack}

We first recall some standard definitions and facts about groupoids
and track categories.

\begin{defn}\label{dgpoid}
Recall that a \emph{groupoid} is a small category $G$ in which all morphisms
are isomorphisms. As for any category, it can be described by a diagram of
sets:
\mydiagram[\label{eqgpoid}]{
G_{1}\times_{G_{0}}G_{1}\ar@<2ex>[rr]^{d_{0}}\ar[rr]^{c}\ar@<-1.5ex>[rr]_{d_{2}}&&
G_{1}\ar@<0.7ex>[rr]^{s} \ar@<-.7ex>[rr]_{t} \ar@/^2pc/[ll]^{s_{1}}
\ar@/_2pc/[ll]_{s_{0}} &&  G_{0} \ar@/_1.3pc/[ll]_{i}
}
\noindent where \w{G_{0}} is the set of objects of $G$ and \w{G_{1}} the
set of arrows. Here $s$ and $t$ are the source and target functions, $i$
associates to an object its identity map, \w{d_{0}} and \w{d_{2}} are
the respective projections, with ``inverses'' \w{s_{0}} and \w[,]{s_{1}}
and $c$ is the composition, satisfying the appropriate identities.

We can think of \wref{eqgpoid} as the $2$-skeleton of a simplicial
set (with \w[,]{G_{2}:=G_{1}\times_{G_{0}}G_{1}} \w[,]{d_{1}=c:G_{2}\to G_{1}}
and so on). The \emph{nerve} functor \w{N:\Gpd\to\Ss} (cf.\
\cite{SegC}) assigns to $G$ the corresponding $2$-coskeletal
simplicial set \w[,]{NG} so
\begin{myeq}\label{eqcoskel}
(NG)_{n}~:=~
\underbrace{G_{1}\times_{G_{0}}G_{1}\dotsc G_{1}\times_{G_{0}}G_{1}}_{n}
\end{myeq}
\noindent for all \w[,]{n\geq 2} with face maps determined by the
associativity of the composition $c$.

Given a groupoid $G$ as above, taking the coequalizer of $s$ and $t$
in \wref{eqgpoid} yields the set \w{\pi_{0}G} (which may be identified
with the usual set \w{\pi_{0}NG} of path components of the nerve \wh
that is, the coequalizer of \w[).]{d_{0},d_{1}:NG_{0}\to NG_{1}}

Note that the usual cartesian product yields a monoidal structure on
the category \w{\Gpd} of groupoids (with the trivial groupoid on one
object as the unit).
\end{defn}

\begin{defn}\label{dtrack}
A \emph{track category} is a (small) category enriched in groupoids,
and a \ww{\GO}-category is a track category with object set $\OO$.
We may identify the category \w{\GOC} of all such track categories
with \w[,]{\Gpd(\OC)} the category of internal groupoids in \w[.]{\OC}
\end{defn}

We use the notation \w{\D=(\D_{1}\toto\D_{0})} to indicate that the
track category $\D$ has \w{\D_{0}} as its category of \emph{$0$-cells}
(objects \w[)]{a,b\in\OO} and \emph{$1$-cells} (maps \w[),]{f:a\to b} and
\w{\D_{1}} as its category of $0$-cells and \emph{$2$-cells} (or
\emph{tracks} \w[).]{\xi:f\Ra g} For fixed \w[,]{a,b\in\OO:=\Obj\D}
we denote by \w{\xi\cv\zeta} the \emph{vertical} (internal) composition of
tracks \w{f\xRa{\xi}g\xRa{\zeta}h} in the groupoid \w[.]{\D(a,b)}
The \emph{homotopy category} of $\D$ in \w[,]{\OC} denoted by
\w{\Pi_{0}\D} or \w[,]{\ho\D} is obtained by applying \w{\pi_{0}} to
each groupoid \w[.]{\D(a,b)} This has equivalence classes of $1$-cells
(with respect to the $2$-cells, which are all invertible) as morphisms.

For a groupoid $G$, we let \w{G\ud} denote the \emph{semi-discrete}
groupoid with the same objects as $G$, with \w{G\ud(a,a)=G(a,a)} for each
\w[,]{a\in\Obj G} and \w{G\ud(a,b)=\emptyset} for \w{a\neq b} (i.e., a
disjoint union of groups). This notation extends to track categories.
Given a category \w[,]{\E\in\OC} a \ww{\GE}-category is a
track category $\D$ with \w[.]{\D_{0}=\E}

The nerve functor \w{N:\Gpd\to\Ss} has a left adjoint, which coincides
with the fundamental groupoid functor \w{\hpi:\Ss\to\Gpd} (cf.\
\cite[Chapter 2]{HigC}) when applied to Kan complexes.
Moreover, \w{\hpi} and $N$ induce a one-to-one correspondence between
$1$-types (i.e., isomorphism classes in \w[)]{\ho\Ss_{\leq 1}} and
equivalence classes of groupoids.

Since $N$ commutes with products, it extends to a functor \w{\N:\GOC\to\SOC}
defined by taking the nerve of each groupoid \w{\D(a,b)} for \w[.]{a,b\in\OO}
$\N$ has a left adjoint \w[,]{\cP:\SOC\to\GOC} which is defined for
fibrant \ww{\SO}-categories (see \S \ref{ssoc} below) by applying
\w{\hpi} to each mapping space \wh again, because \w{\hpi} commutes
with products. The fibrancy is needed here because when \w{X\in\Ss} is
not fibrant, the usual construction of \w{\hpi X} involves first
replacing it by a Kan complex (cf.\ \cite[I, \S 8]{GJarS})
Moreover, since the nerve of a groupoid is $2$-coskeletal, $\N$ in
fact lands in the category \w{\PSOC{1}} of ``categories enriched in
$1$-types'' (see \S \ref{snot}), so we have functors:
\mydiagram[\label{eqnervefund}]{
\PSOC{1} \ar@<1ex>[rr]^{\cP} & &\GOC \ar@<1ex>[ll]^{\N} }
\noindent Under the identifications \w{\GOC\cong\Gpd(\OC)} and
\w[,]{\SOC\cong s\OC} the adjoint pair \wref{eqnervefund}
corresponds to the adjunction \w{\Gpd(\OC)\rightleftharpoons s\OC}
between the nerve functor on internal groupoids and its left adjoint.

We say that a morphism $f$ in \w{\TraO} is a \emph{weak equivalence} if
\w{\N f} is a weak equivalence in the standard model structure on
\w{\SOC} (see \S \ref{ssoc} below). If \w{\TraO\bsim} denotes the
localization of \w{\TraO} with respect to weak equivalences, and
similarly for \w{\SOC} or \w[,]{\PSOC{n}}  then  $\cP$ and $\N$
induce equivalences between \w{\PSOC{1}\bsim} and \w[.]{\TraO\bsim}

%
%
\section{Two-typical double groupoids}
\label{cdgpd}

As noted in the Introduction, there are several algebraic structures
which model $2$-types. We now describe a certain kind of double
groupoid which can be used as such models. These are equipped with a
pair of adjoint functors, which enables us to pass back and forth
between \ww{\Po{2}}-simplicial sets and such double groupoids.

\begin{defn}\label{ddgpd}
A \emph{double groupoid} (cf.\ \cite{EhreCD}) is a groupoid internal
to \w[:]{\Gpd} in other words, a diagram of the form \wref[,]{eqgpoid}
with \w{G_{0}} and \w{G_{1}} in \w[,]{\Gpd} rather than \w[.]{\Set}
A double groupoid \w{\Gss} may thus be described explicitly by a
diagram (of sets) of the form:
\mydiagram[\label{eqdgpd}]{
& G_{11}\times_{G_{10}}G_{11}\ar[d]^{c^{1\ast}} \ar@<0.5ex>[r]\ar@<-0.5ex>[r] &
G_{01}\times_{G_{00}}G_{01}\ar[d]^{c^{0\ast}} \\
G_{11}\times_{G_{01}}G_{11}\ar[r]^<<<<<<<{c^{1 \ast}}
\ar@<0.5ex>[d]\ar@<-0.5ex>[d] &
G_{11} \ar@<0.5ex>[d]^{d_{0}^{1\ast}} \ar@<-0.5ex>[d]_{d_{1}^{1\ast}}
\ar@<0.5ex>[r]^{d_{0}^{\ast 1}} \ar@<-0.5ex>[r]_{d_{1}^{\ast 1}}  &
G_{01} \ar@<0.5ex>[d]^{d_{1}^{0\ast}} \ar@<-0.5ex>[d]_{d_{1}^{0\ast}}
\ar@/_1.5pc/[l] \\
G_{10}\times_{G_{00}}G_{10} \ar[r]^<<<<<<<{c^{\ast 0}} &
G_{10} \ar@<0.5ex>[r]^{d_{0}^{\ast 0}} \ar@<-0.5ex>[r]_{d_{1}^{\ast 0}}
\ar@/^2pc/[u] &
G_{00} \ar@/^1.5pc/[l] \ar@/_2pc/[u]
}
\noindent satisfying certain axioms. Here each of the four pairs of
maps marked \w{d_{0}} and \w{d_{1}} are source and target maps for one
of the four groupoids, and the curved maps are the identities.
We think of \w{G_{\ast 0}} and \w{G_{\ast 1}} as
the \emph{horizontal} groupoids of \w[,]{\Gss} and \w{G_{0\ast}} and
\w{G_{1\ast}} as the \emph{vertical} ones. The category of double
groupoids is denoted by \w[.]{\TGpd}
\end{defn}

\begin{defn}\label{ddgnerve}
By applying the nerve functor \w{N:\Gpd\to\Ss} horizontally to a
double groupoid \w[,]{\Gss\in\TGpd} we obtain a simplicial groupoid
\w[.]{N^{h}\Gss\in s\Gpd} If we then apply $N$ again vertically, we
obtain a bisimplicial set \w[,]{N^{v}N^{h}\Gss\in s\Ss} called the
\emph{double nerve} of \w[.]{\Gss} Taking its diagonal yields the
\emph{diagonal nerve} of \w[,]{\Gss} so that the functor
\w{\dN:\TGpd\to\Ss} is the composite of:
$$
\TGpd~\xra{N^{h}}~s\Gpd~\xra{N^{v}}~s\Ss~\xra{\diag}~\Ss~.
$$
\end{defn}

From general categorical considerations it is clear that
\w{\dN:\TGpd\to\Ss} must have a left adjoint; however, it
is hard to describe this adjoint explicitly in a useful way. We now
define a functor \w[,]{\dP:\Ss\to\TGpd} equivalent up to homotopy to
this adjoint, which has a particularly simple form when \w{X\in\Ss} is
fibrant, and which takes values in a convenient subcategory of
\w[.]{\TGpd} We shall use this construction of \w{\dP X} as our
canonical double groupoid model for $X$. In fact, \w{\dP} will be the
composite of two functors, so we start with the following:

\begin{defn}\label{dos}
Let the functor \w{\os:\Ss\to s\Ss} be induced by the ordinal sum
\w{\orr:\Delta^{2}\to\Delta} (where the category $\Delta$ of finite
ordinals is the indexing category for simplicial objects). Given \w[,]{X\in\Ss}
the bisimplicial set \w{\os X} can be described as ``total \ww{\Dec}''
of \cite{IlluC2} (see also \cite{DuskS}), as follows:

Let \w{\Aug\Ss} denote the category of augmented simplicial sets. There is a
functor \w{\Dec:\Ss\to\Aug\Ss} which forgets the last face
operator. This has a right adjoint \w[,]{+:Aug\Ss\to\Ss} which forgets
the augmentation. The adjoint pair \w{(\Dec,+)} gives rise to a
comonad, and the resulting comonad resolution of $X$ is \w[.]{\os X}

Explicitly, this has the form:
$$
\xymatrix@R=25pt{
    \cdots \ssr & X_5 \ar@<1ex>[r] \ar[r] \ar@<-1ex>[r] \ar@<1ex>[d] \ar[d]
    \ar@<-1ex>[d]& X_4 \ar@<1ex>[r] \ar[r]  \ar@<1ex>[d] \ar[d] \ar@<-1ex>[d] &
    X_3 \ar@<1ex>[l] \ar@<1ex>[d] \ar[d] \ar@<-1ex>[d]\\
    \cdots \ssr  & X_4 \ar@<1ex>[r] \ar[r] \ar@<-1ex>[r] \ar@<-1ex>[d] \ar[d]  &
    X_3 \ar@<1ex>[r] \ar[r]   \ar[d] \ar@<-1ex>[d] &
    X_{2} \ar@<1ex>[l] \ar@<-1ex>[d] \ar[d] \\
    \cdots \ssr & X_3 \ar@<1ex>[r] \ar[r]\ar@<-1ex>[r] \ar@<-1ex>[u] &
     X_{2} \ar@<1ex>[r] \ar[r] \ar@<-1ex>[u]  &
    X_{1} \ar@<1ex>[l]\ar@<-1ex>[u]
}
$$
\end{defn}

The following fact is straightforward:

\begin{lemma}\label{lfibdec}
Let $X$ be a fibrant simplicial set. Then \w{\Dec X} is also fibrant and the
map \w{\partial:\Dec X\to X} is a fibration.
\end{lemma}

\begin{defn}\label{dkanp}
The functor \w{\dP:\Ss\to\TGpd} is defined to be the composite
\w[,]{\dP:=\hP\circ\os} where $\hP$ is the left adjoint to the double
nerve functor \w[.]{N^{v}N^{h}:\TGpd\to ss\Set}
\end{defn}

\begin{remark}\label{rintcase}
The internal case (in groups) recovers the fundamental
\ww{\Cat^{2}}-group functor of a simplicial group (see \cite{BCDusC}).
\end{remark}

In general, it is hard to describe the left adjoint
\w{\hP:ss\Set\to\TGpd} in a useful way. In order to get a simple
description in certain cases, we need the following:

\begin{defn}\label{dcsktfib}
We say that a simplicial set $X$ is \ww{\csk{2}}-\emph{fibrant}
if \w{\csk{2} X} is fibrant. This is equivalent to saying that
\w{\Lambda^k[n]\to X} has a filler \w{\Delta[n]\to X} for each
\w[.]{0< n\leq 2} Similarly, we say that \w{f:X\to Y} is a
\ww{\csk{2}}-\emph{fibration} if \w{\csk{2} f} is a fibration.
\end{defn}

\begin{remark}\label{rcsktfib}
Recall that when \w{X\in\Ss} is fibrant, its fundamental groupoid
\w{\hpi X} has a particularly simple description: its set of objects
is \w[,]{X_{0}} and for \w[,]{x,x'\in X_{0}} the morphism set
\w{(\hpi X)(x,x')} is
\w[,]{\{\tau\in X_{1}~:\ d_{0}\tau=x,d_{1}\tau=x'\}\bsim} where $\sim$
is determined by the $2$-simplices of $X$. We write \w{(\hpi X)_{1}}
for \w[.]{X_{1}\bsim}

Note that \w{\hpi:\Scf\to\Gpd} factors through \w[,]{\csk{2}} so
this description is valid for any \ww{\csk{2}}-fibrant $X$.
Note also that if \w{\Xd} is a simplicial groupoid,
\w{d_{1}:X_{n1}\to X_{n0}} is the target map of the groupoid \w{X_{n}}
for each \w[.]{n\geq 0}
\end {remark}

%
%
\begin{prop}\label{pone}
Let \w{\Xd\in s\Gpd} be a (horizontal) simplicial groupoid, for which
the simplicial sets \w{X\bull_{0}} and \w{X\bull_{1}} are
\ww{\csk{2}}-fibrant, and
the morphism \w{d_{1}^{v}:X\bull_{1}\to X\bull_{0}} is a
\ww{\csk{2}}-fibration.
Then the left adjoint \w{\Ph:s\Gpd\to\TGpd} to the nerve
\w[,]{N^{h}:\TGpd\to s\Gpd} applied to \w[,]{\Xd} is \w[.]{\hpi^{h}\Xd}
\end{prop}

%
%
\begin{prop}\label{ptwo}
Let \w{\Xdd\in s\Ss} be such that \w{X\bull_{i}} and \w{X_{i}\bull} are
\ww{\csk{2}}-fibrant for each \w[,]{i\geq 0} and
\w{d_{0}^{h}:X_{1}\bull\to X_{0}\bull} and
\w{d_{0}^{v}:X\bull_{1}\to X\bull_{0}} are \ww{\csk{2}}-fibrations. Then
\begin{enumerate}
\renewcommand{\labelenumi}{(\roman{enumi})\ }
\item \w{(N^{v}\hpi^{v}\Xdd)_{i}\bull} is fibrant for all \w[.]{i\geq 0}
\item \w{(N^{v}\hpi^{v}\Xdd)\bull_{1}} is \ww{\csk{2}}-fibrant.
\item \w{\ovl{d}^{v}_{0}:(N^{v}\hpi^{v}\Xdd)\bull_{1}\to
      (N^{h}\hpi^{v}\Xdd)\bull_{0}} is a \ww{\csk{2}}-fibration.
\item \w{\ovl{d}^{h}_{0}:(N^{v}\hpi^{v}\Xdd)_{1}\bull\to
        (N^{v}\hpi^{v}\Xdd)_{0}\bull} is a fibration.
\end{enumerate}
\end{prop}

For the proofs of these Propositions, see Appendix A.

%
%
\begin{cor}\label{cladoub}
If \w{\Xd\in s\Ss} satisfies the hypotheses of Proposition \ref{ptwo},
then \w[,]{\hP\Xd=\hpi^{h}\hpi^{v}\Xd} and thus for
a Kan complex \w{X\in\Scf} we have \w[.]{\dP X=\hpi^{h}\hpi^{v}\os X}
\end{cor}

\begin{remark}\label{ramazur}
The functor \w{\dP} is not actually the left adjoint of \w[.]{\dN} However,
if we replace \w{\diag:ss\Set\to\Ss} by the homotopy equivalent functor
\w{\ovl{W}:ss\Set\to\Ss} of \cite[\S III]{AMaV}, which is right adjoint to
\w[,]{\os} then \w{\dP} will be left adjoint to
\w{\ovl{W}N^{v}N^{h}} (see \cite{CRemeD}).
\end{remark}

\begin{mysubsect}{Two-typical double groupoids}
\label{swgdg}

We now have a natural choice for modeling a space $Y$ by a double
groupoid: choose a Kan complex $X$ weakly equivalent to $Y$, and
take \w[.]{\dP X\in\TGpd} It turns out that the double groupoids
obtained in this way have several convenient properties. First of
all, it is not hard to see that \w{\dP X} is symmetric, in the
following sense:
\end{mysubsect}

\begin{defn}\label{dspecialdg}
A double groupoid \w{\Gss} is called \emph{symmetric} if
\w{G_{0\ast}\cong G_{\ast 0}} and \w{G_{1\ast}\cong G_{\ast 1}} are
isomorphic groupoids. In this case \w{\Gss} may be described more
succinctly by a diagram
\mydiagram[\label{eqwkglgpd}]{
G_{[1]}\ar@<0.7ex>[rr]^{d_{0}} \ar@<-.7ex>[rr]_{d_{1}}  &&
G_{[0]}\ar@/_1.7pc/[ll]_{s}
}
\noindent in which
\mydiagram[\label{eqwkglgp}]{
G_{[1]}=(G_{2}\ar@<0.7ex>[rr]^{d_{0}^{[1]}} \ar@<-.7ex>[rr]_{d_{1}^{[1]}}  &&
G_{1})\ar@/_1.7pc/[ll]_{s^{[1]}}& \text{and} &
G_{[0]}=(G_{1}\ar@<0.7ex>[rr]^{d_{0}^{[0]}} \ar@<-.7ex>[rr]_{d_{1}^{[0]}}
&&  G_{0}) \ar@/_1.7pc/[ll]_{s^{[0]}}
}
are groupoids (isomorphic to \w{G_{1\ast}} and \w[,]{G_{0\ast}}
respectively). To make this precise, one should expand
\wref{eqwkglgp} to a diagram of the form \wref[,]{eqdgpd} satisfying
suitable axioms.

Applying \w{\pi_{0}} to \w{G_{[1]}} and \w{G_{[0]}} yields the following
diagram of sets:
\mydiagram[\label{eqfundgpd}]{
\pi_{0}G_{[1]} \ar@<1ex>[r]\ar@<-1ex>[r] & \pi_{0}G_{[0]}
}
where the two maps are induced by \w[.]{d_{0},d_{1}: G_{[1]}\to G_{[0]}}
This is equivalent to applying the functor \w{\pi_{0}} either
horizontally or vertically to \w{\Gss} in the diagram \wref{eqdgpd}
(or equivalently, to the symmetric bisimplicial set \w[)\vsm.]{N^{v}N^{h}\Gss}
\end{defn}

In addition, Propositions \ref{pone} and \ref{ptwo} imply that \w{\dP X}
satisfies certain fibrancy conditions:

\begin{defn}\label{dweakglob}
For any groupoid $G$,  let \w{G^{d}} denote the discrete groupoid
  on the set \w[.]{\pi_{0}G} This comes equipped with a map of groupoids
  \w[.]{\gamma:G\to G^{d}}

A double groupoid \w{\Gss} is called \emph{weakly globular} if
\begin{enumerate}
\renewcommand{\labelenumi}{(\alph{enumi})\ }
\item The map \w{\gamma:G_{0\ast}\to G_{0\ast}^{d}} is a weak equivalence.
\item Both \w{N^{v}d^{h}_{0}} and \w{N^{v}d^{h}_{1}} are fibrations of
  simplicial sets, where \w{d^{h}_{0},d^{h}_{1}: G_{1\ast}\to G_{0\ast}}
  are maps of (vertical) groupoids.
\end{enumerate}
\end{defn}

\begin{remark}\label{rweakglob}
Note that a strict $2$-groupoid (i.e., a groupoid enriched in
groupoids \wh cf.\ \cite[\S 7.7]{BorcH1}) is an example of a weakly
globular double groupoid, since a map of simplicial sets with discrete
target is a fibration. The internal version of this concept (in
groups) is the notion of weakly globular \ww{\cat^{2}}-group of
\cite{PaolW} (in this case the fibrancy conditions are automatically
satisfied).
\end{remark}

\begin{defn}\label{dtwotypical}
A double groupoid is called \emph{two-typical} if it is symmetric and
weakly globular. The full subcategory of \w{\TGpd} whose objects are
two-typical double groupoids will be denoted by \w[.]{\TGw}
\end{defn}

%
%
\begin{lemma}\label{lfibr}
Let \w{p:\G\to\G'} be a map of groupoids. Suppose  that the
    diagram:
$$
\xymatrix{
        \Lambda^k[n]\ar^{f}[r]\ar@{^{(}->}_{i}[d] &  N\G \ar^{N p}[d]\\
        \Delta[n]\ar_{h}[r] \ar@{.>}[ru] &  N\G'
        } \\
$$
has a lift \w{\Delta[n]\to N\G} when \w[.]{n=1} Then \w{N p} is a fibration.
\end{lemma}

\begin{proof}
This follows because the nerve of a groupoid is $2$-coskeletal.
\end{proof}

We now show that the functor \w{\dP:\Ss\to\TGpd} of Definition
\ref{dkanp} indeed takes values in \w{\TGw} (when applied to a Kan
complex):

\begin{prop}\label{pthree}
If $X$ is a fibrant simplicial set, then \w{\dP X} is a two-typical
double groupoid.
\end{prop}

\begin{proof}
Let $X$ be a fibrant simplicial set.  Evidently \w{\dP X} is
symmetric (that is, \w{\dP X_{i,j}\cong\dP X_{j,i}} for all
\w[),]{i,j} since \w{\os X} is. Recall that for \w[:]{i>0}
\begin{equation*}
    \begin{split}
      (\os X)_{0\ast}= & \Dec X=(\os X)_{\ast0} \\
      (\os X)_{i\ast}= & \Dec(\os X)_{i-1,\ast} \\
      (\os X)_{\ast i}= & \Dec(\os X)_{\ast,i-1}~.
    \end{split}
\end{equation*}

By Lemma \ref{lfibdec}, it follows that \w{\os X} satisfies the hypotheses of
Proposition \ref{ptwo}. Hence, by Proposition \ref{ptwo} we see
that \w{\hpi^{v}\os X\in s\Gpd} satisfies the hypotheses of Proposition
\ref{pone}. Thus if \w{\Ph} is the left adjoint to \w[,]{N^{h}}
we see \w{\Ph\hpi^{v}\os X} is computed by applying \w{\hpi^{h}}
levelwise in the horizontal direction in the bisimplicial set
\w[.]{N^{v}\hpi^{v}\os X} That is:
\begin{myeq}\label{eqhvp}
\hP\os X~=~\Ph\hpi^{v}\os X~=~\hpi^{h}\hpi^{v}\os X
\end{myeq}

Furthermore, by Proposition \ref{ptwo} the bisimplicial set
\w{N^{v}\hpi^{v}\os X} itself satisfies the hypotheses
of Proposition \ref{ptwo}; therefore we conclude by Proposition
\ref{ptwo} iii) that \w{N d_{0}^{v}} is a fibration, for
\w[.]{d_{0}^{v}:(\hP\os X)_{\ast1}\to(\hP\os X)_{\ast0}}

To show that \w{\hP\os X} is weakly globular, it remains to show that
the vertical groupoid \w{(\hP\os X)_{0\ast}=\hpi\Dec X} is equivalent
to a discrete groupoid. Recall that if we let \w{c(X_{0})} denote the
constant simplicial set on the set \w[,]{X_{0}} there are maps
\w{c(X_{0})\supar{s}\Dec X\supar{v}c(X_{0})} with \w[,]{v\circ s=\Id}
and there is a simplicial homotopy equivalence \w[.]{v\circ s\simeq\Id}
Thus the simplicial sets \w{c(X_{0})} and \w{\Dec X} are weakly
equivalent and therefore have equivalent fundamental groupoids. There
are thus equivalences of groupoids \w{r:(\hP\os X)_{0\ast}\to\hpi c(X_{0})}
and \w[,]{t:\hpi c(X_{0})\to(\hP\os X)_{0}} with \w[,]{r\circ t=\Id} and
\w{\hpi c(X_{0})} is the discrete groupoid on the set \w[.]{X_{0}}
\end{proof}

%
%
\section{Two-typical double groupoids and $2$-types}
\label{cdgttype}

We now show that a two-typical double groupoid \w{\Gss} is a
$2$-type, in the sense that its diagonal nerve, the simplicial set
\w[,]{\dN\Gss} is in \w{\Po{2}\Ss}  (cf.\ \S \ref{snot}).
Actually, we prove a little more:

%
%
\begin{prop}\label{pfour}
For any weakly globular \w[,]{\Gss\in\TGpd} the realization
\w{X:=\dN\Gss\in\Ss} is a $2$-type.
\end{prop}

\begin{proof}
Since \w{\dN\Gss} is the diagonal of the bisimplicial set
\w[,]{N^{v}N^{h}\Gss} it suffices to find a simplicial groupoid
\w{\Yd} such that \w{\eN\Yd:=\diag N^{v}\Yd} is a $2$-type, together
with a map of simplicial groupoids \w{N^{h}\Gss\to\Yd} which is a
weak equivalence in each simplicial dimension, so that the induced map
\begin{myeq}\label{eqdiagequiv}
\dN\Gss=\diag N^{v}N^{h}\Gss~\xra{\simeq}~\eN\Yd
\end{myeq}
is a weak equivalence, too.

Recall that a \emph{Tamsamani weak $2$-groupoid}
is a simplicial groupoid \w{\Yd\in s\Gpd} such that:
\begin{enumerate}
\renewcommand{\labelenumi}{(\alph{enumi})\ }
\item \w{Y_{0}} is discrete;
\item The Segal maps
\begin{myeq}\label{eqsegal}
\eta_{n}:Y_{n}~\to~
\underbrace{Y_{1}\times_{Y_{0}}Y_{1}\dotsc Y_{1}\times_{Y_{0}}Y_{1}}_{n}
\end{myeq}
are equivalences of groupoids (that is, \w{N\eta_{n}} is a weak
equivalence). The full subcategory of such objects in \w{s\Gpd} is denoted
by \w[.]{\Tam}
\item The simplicial set \w[,]{\pi_{0}\Yd} obtained by applying
  \w{\pi_{0}} in each simplicial dimension, is the nerve of a groupoid.
\end{enumerate}

In this case, the simplicial set \w{\eN\Yd} is in \w[.]{\Po{2}\Ss} Moreover,
simplicial sets of the form \w{\eN\Yd} model all $2$-types of simplicial sets,
up to homotopy (see \cite{TamsN}).

The idea of the construction of $\Yd$ is to make $\Zd$ into a globular
object by replacing the groupoid $Z_0=G_{0*}$ with the discrete groupoid
${G^d_{0*}}$, in a way that does not change the homotopy type of $\Zd$. By
pushing out the unique degeneracy map $s_{(n)}:Z_0\to Z_n$ along the
discretization map $\gamma:G_{0*}\to G^d_{0*}$ one obtains a simplicial
groupoid $\Yd$; since $\gamma$ is an equivalence, it is easily seen that the
simplicial map $\Zd\to\Yd$ is a levelwise equivalence. This, together with the
fact that the Segal maps in $\Zd$ are isomorphisms and the maps $Nd_0, Nd_1:
NZ_1\to NZ_0$ are fibrations also implies that the Segal maps of $\Yd$ are
equivalences. In conclusion, $\Yd$ will be a Tamsamani weak 2-groupoid with the
same homotopy type as $\Zd$. The details are as follows:

Let \w{\Zd=N^{h}\Gss} and \w{Y_{0}:=G_{0\ast}^{d}}  (see Definition
\ref{dweakglob}).
Note for each \w[,]{n\geq 0} there is a unique morphism
\w{s_{(n)}:\bze\to\bn} in \w[.]{\Delta\op} We use the same
notation \w{s_{(n)}:W_{0}\to W_{n}} for any simplicial object \w[.]{\Wd}
For each \w[,]{n>0} \w{Y_{n}} is defined to be the pushout in \w[:]{\Cat}
\begin{equation*}
    \xymatrix{
    G_{0\ast}=:Z_{0} \ar^<<<<<<<<<{s_{(n)}}[rr] \ar^{\gamma=:f_{0}}[d] &&
    Z_{n} \ar^{f_{n}}[d]\\
    G_{0\ast}^{d}=:Y_{0} \ar^<<<<<<<<<{\sigma_{(n)}}[rr] && Y_{n}
    }
\end{equation*}
We claim that \w{Y_{n}} is a groupoid and \w{f_{n}} is a weak
equivalence. In fact, since \w{\Gss} is weakly globular, $\gamma$ is an
equivalence of groupoids, and thus also a categorical equivalence. The
map \w{s_{(n)}} is injective on objects. But the pushout in \w{\Cat}
of a categorical equivalence by a map which is injective on objects is
a categorical equivalence (see \cite{JStreP}); hence \w{f_{n}} is a
categorical equivalence, and thus a weak equivalence. Since a category
equivalent to a groupoid is itself a groupoid, \w{Y_{n}} is a
groupoid. This proves the claim.

Let \w{\phi:\bn\to\bm} be any morphism in \w[.]{\Delta\op}
Then \w[,]{\phi s_{(n)}=s_{(m)}} by the uniqueness, so that
\w[.]{f_{m}\phi s_{(n)}=f_{m}s_{(m)}=\sigma_{(m)}f_{0}:Z_{0}\to Y_{m}}
From the universal property of pushouts there is thus a unique map
\w{\hat{\phi}:Y_{n}\to Y_{m}} such that \w{\hat{\phi} f_{n}=f_{m}\phi} and
\w[.]{\hat{\phi}\sigma_{(n)}=\sigma_{(m)}} In particular we have maps
\w{\hat{\partial}_{i}:Y_{n}\to Y_{n-1}} for \w{0\leq i\leq n} and
\w{\hat{\sigma}_{i}:Y_{n-1}\to Y_{n}} for \w[.]{0\leq i< n}

The maps \w{\hat{\partial}_{i}} and \w{\hat{\sigma}_{i}}
\wb{0\leq i\leq n<\infty} satisfy the simplicial identities, so that
these make \w{\Yd=(Y_{n})_{n=0}^{\infty}} into a simplicial groupoid.
To see this, let \w[,]{\phi:\bn\to\bm} and \w{\psi:\bm\to\bk} be any
morphisms in \w{\Delta\op} with \w[.]{\xi:=\phi\circ\psi} Then:
\begin{equation*}
%
\hat{\xi}\sigma_{(n)}=\sigma_{k)}~=~
\hat{\phi}_{1}\sigma_{(m)}=\hat{\phi}_{1}\hat{\phi}\sigma_{(n)}\hsp\text{and}\hsp
\hat{\xi}f_{n}=f_{k}\xi~=~
f_{k}\psi\phi=\hat{\phi}' f_{m}\phi=\hat{\psi}'\hat{\phi} f_{n}~.
%
\end{equation*}
It follows by universality of pushouts that
\w[.]{\hat{\xi}=\hat{\psi}'\hat{\phi}} In particular, since the
simplicial identities are satisfied by the maps \w{\partial_{i}} and
\w[,]{\sigma_{i}} they are satisfied by \w{\hat{\partial}_{i}} and
\w[.]{\hat{\sigma}_{i}}

We now prove that \w{\Yd\in s\Gpd} is a Tamsamani weak $2$-groupoid. By
construction, \w{Y_{0}=N^{v}G_{\ast 0}^d} is discrete. We need to prove that,
for
each \w[,]{n\geq 2} the Segal maps \wref{eqsegal} are equivalences of groupoids.

Consider the case \w[.]{n=2} There is a commutative diagram in \w[:]{\Ss}
\begin{myeq}\label{eqgpdopn}
\begin{split}
&
\xymatrix{N Z_{2} \ar@{=}[rr] \ar_{f_{2}}[d] && \tens{N Z_{1}}{N Z_{0}}
\ar^{f_{1}\times f_{1}}[d]\\
N Y_{2} \ar_{\eta_{2}}[rr] && \tens{N Y_{1}}{N Y_{0}}
}
\end{split}
\end{myeq}
We claim that \w{f_{1}\times f_{1}} is a weak equivalence. In fact, there is a
commutative diagram in $\Ss$:
$$
\xymatrix{
N Z_{1} \ar^{N d_{0}}[rr] \ar_{N f_{1}}[d] && N Z_{0} \ar^{N f_{0}}[d]
    && N Z_{1} \ar_{N d_{1}}[ll] \ar^{N f_{1}}[d]\\
    N Y_{1} \ar_{N d'_{0}}[rr] && NY_{0} && NY_{1} \ar^{Nd'_{1}}[ll]
    }
$$
In this diagram, the map in each column is a weak equivalence; the map
\w{N d_{0}} is a fibration since \w{\Gss} is weakly globular. Since \w{N Y_{0}}
is constant, the map \w{N d'_{0}} satisfies the hypotheses of Lemma
\ref{lfibr}, and is therefore a fibration. Since the standard
model structure on $\Ss$ is right proper, we can apply
\cite[Proposition 13.3.9]{PHirM} to conclude that the induced map of
pullbacks \w{f_{1}\times f_{1}} is a weak equivalence, as claimed.

From above, we know that \w{f_{2}} is also a weak equivalence. The
commutativity of \wref{eqgpdopn} therefore implies that \w{\eta_{2}}
is a weak equivalence. Similarly one shows that \w{\eta_{n}} is a weak
equivalence for each \w[.]{n>2} To show that \w{\Yd} is a Tamsamani
weak $2$-groupoid, it remains to check that \w{\pi_{0}\Yd} is the nerve of a
groupoid. For this, notice that there is a commutative diagram in $\Ss$:
$$
\xymatrix{
N Z_{1} \ar^{N d_{0}}[rr] \ar@{=}_{}[d] && N Z_{0} \ar^{N f_{0}}[d] &&
N Z_{1} \ar_{N d_{1}}[ll] \ar@{=}^{}[d]\\
N Z_{1} \ar_{N dd_{0}}[rr] && N Y_{0} && N Y_{1} \ar^{N dd_{0}}[ll]
}
$$
Again each vertical map is a weak equivalence and each horizonal map
is a fibration. Thus we conclude that the induced map of pullbacks
$$
N(\tens{Z_{1}}{Z_{0}})=\tens{N Z_{1}}{N Z_{0}}
\to\tens{N Z_{1}}{N Y_{0}}= N(\tens{Z_{1}}{Y_{0}})
$$
is a weak equivalence. Thus the map of groupoids
\w{\tens{Z_{1}}{Z_{0}}\to\tens{Z_{1}}{Y_{0}}} is an equivalence of
categories. Since the functor \w{\pi_{0}:\Gpd\to\Set} preserves fibre
products over discrete objects, we have an isomorphism
\begin{equation*}
\begin{split}
(\pi_{0}\Zd)_{n}~\cong ~
\pi_{0}(\underbrace{Z_{1}\times_{Z_{0}}\dotsc
\times_{Z_{0}}Z_{1}}_{n})
~\cong&~(\underbrace{\pi_{0}Z_{1}\times_{\pi_{0}Z_{0}}\dotsc
\times_{\pi_{0}Z_{0}}\pi_{0}Z_{1}}_{n})\\~\cong&~
\underbrace{(\pi_{0}\Zd)_{1}\times_{(\pi_{0}\Zd)_{0}}\dotsc
\times_{(\pi_{0}\Zd)_{0}}(\pi_{0}\Zd)_{1}}_{n}~.
\end{split}
\end{equation*}
This shows that the simplicial set \w{\pi_{0}\Zd} has all Segal maps
isomorphisms, and is therefore the nerve of a category. In fact, since
\w{\Gss} is a double groupoid, \w{\pi_{0}\Zd} is the nerve of a groupoid.

Since for each \w[,]{n\geq 0} \w{f_{n}:Z_{n}\to Y_{n}} is a weak
equivalence, \w[.]{(\pi_{0}\Yd)_{n}=(\pi_{0}\Zd)_{n}}
Hence \w[,]{\pi_{0}\Yd\cong\pi_{0}\Zd} so that from above
\w{\pi_{0}\Yd} is also the nerve of a groupoid, as required.
This actually defines a functor \w{T:\TGw\to\Tam} with \w[.]{T(\Gss)=\Yd}

This concludes the proof that \w{\Yd} is a Tamsamani weak $2$-groupoid. From
\cite{TamsN}, we know that \w{\eN\Yd} is a $2$-type. Since the map
\w{f:N^{v}\Gss=\Zd\to\Yd} is a levelwise weak equivalence, the map
\w{N^{h} f} is a levelwise weak equivalence in \w[,]{s\Ss} and
therefore \w{\dN f} is a weak equivalence. It follows by
\wref{eqdiagequiv} that \w[,]{\dN\Gss~\simeq~\eN\Yd} so that
\w{\dN\Gss} is a $2$-type.
\end{proof}

\begin{remark}\label{rpost}
We actually can read off more information from a two-typical double
groupoid \w[,]{\Gss} than just its $2$-type. Namely, we can describe
algebraically its Postnikov decomposition. This will be useful later, in
defining the homotopy track category of a $2$-track category (see \ref{shtc}) and
the notion of  $2$-track extension (see \ref{dbwcc}).
 First, observe that in the proof of Proposition \ref{pfour} we have
 shown that if \w{\Gss} is any weakly globular double groupoid, then
 the diagram \wref[,]{eqfundgpd} obtained by applying the coequalizer
 \w{\pi_{0}} (horizontally or vertically) to \w{\Gss} itself has the
 structure of a groupoid, which we call the \emph{fundamental groupoid} 
of \w[,]{\Gss} and denote by \w[.]{\hpi\Gss} 

Also, there is a simplicial map \w[,]{\Zd=\N^{h}\Gss\to c\hpi\Gss} where
\w{c\hpi\Gss} denotes the constant simplicial groupoid on
\w[.]{\hpi\Gss} Likewise, applying $\pi_0$ in each simplicial
dimension to the simplicial groupoid \w{\Yd} yields the nerve of a
groupoid \w[.]{\hpi\Yd} By the proof of \ref{pfour}, the map 
\w{\Zd\to\Yd} induces an isomorphism \w[.]{\hpi\Gss\cong\hpi\Yd}
Further, by \cite{TamsN}, \w{\pi_{0}\hpi\Yd=\pi_{0}B\Yd} and
\w[,]{\pi_{1}\hpi\Yd(\Id_{\ast})=\pi_{1}(B\Yd,\ast)} where
\w{B=\diag\circ\ovl{\N}} is the realization functor. 

Since \w{\Zd} and \w{\Yd} have the same homotopy types, we conclude
that the map \w{\Zd\to c\hpi\Gss} induces isomorphisms of homotopy
groups in dimension $0$ and $1$. Hence this map gives the last stage
of the Postnikov decomposition of \w[.]{\Gss} In Theorem
\ref{tmodeltt}, we will show that, given \w[,]{X\in\St} \w{\dP X}
represents the $2$-type of $X$. The algebraic description of the
Postnikov decomposition of \w{\dP X} given above translates into the
fact that, for each \w[,]{g\in G_{0}} the fundamental group of the
groupoid \w{G_{[1]}=G_{2}\toto G_{1}} based at \w{\Id_{g}} (in the
notation of \S \ref{dspecialdg}) is isomorphic to the local system
\w[.]{\pi_{2}(X,[g])}  Thus we can actually recover the Postnikov
system of a $2$-type $X$ algebraically from its two-typical model 
\w[\vsm.]{\dP X}
\end{remark}

Proposition \ref{pfour} shows that the functor \w{\dN:\TGw\to\Ss}
takes values in \w{\St} (see \S \ref{snot}) \wh i.e., the realization
of a two-typical double groupoid is a $2$-type.  As before, we say
that a map $f$ in \w{\TGpd} is a \emph{weak equivalence} if \w{\dN f} is a weak
equivalence in $\Ss$. We now show that we have a one-to-one
correspondences of weak equivalence classes of objects of \w{\TGpd}
and \w[:]{\St}

%
%
\begin{thm}\label{tmodeltt}
The functors \w{\dP:\St\to\TGw} and \w{\dN:\TGw\to\St}
induce equivalences of categories after localization:
\mydiagram[\label{eqeqloc}]{
\St\bsim\ar@<1ex>[r]^{\dP} &\TGw\bsim~.\ar@<1ex>[l]^{\dN}
}
\end{thm}

To show this, we shall need the following concept:

\begin{defn}\label{dndiag}
A map \w{f:\Wdd\to\Vdd} of bisimplicial sets is called a
\emph{diagonal $n$-equivalence} if \w{f_{k}^{h}:W_{k}^{h}\to V_{k}^{h}} is an
\ww{(n-k)}-equivalence for each \w[.]{k\leq n}
\end{defn}

%
%
\begin{prop}\label{pndiag}
If \w{f:\Wdd\to\Vdd} is an diagonal $n$-equivalence, then the induced map
\w{\diag f:\diag\Wdd\to\diag\Vdd} is an $n$-equivalence.
\end{prop}

For the proof, see Appendix B.

\begin{proof}
By \cite[I, \S 1]{QuiH} we must show that for any Kan complex
\w{X\in\St} there is a weak equivalence \w[,]{X\simeq\dN\dP X}
where \w{\dP X=\hpi^{h}\hpi^{v}\os X} (cf.\ \S \ref{dkanp}), and that the
natural map \w{\os X\to N^{v}\hpi^{v}\os X} is a weak equivalence in each
(horizontal) simplicial dimension.
Moreover, since \w{(\os X)_{i}\bull} is homotopically trivial for
each \w[,]{i\geq 0} we have
\w[,]{(\os X)_{i}\bull\simeq N^{v}\hpi^{v}(\os X)_{i}\bull}
so
\begin{myeq}\label{eqdiagwe}
\diag \os X~\simeq~\diag N^{v}\hpi^{v}\os X~.
\end{myeq}
\noindent Furthermore, if \w{c(X)_{i}\bull} is the constant
simplicial set on \w[,]{X_{i}} the augmentation map of bisimplicial
sets \w{\os X\to c(X)} is a weak equivalence in each (horizontal)
simplicial dimension, so it induces a weak equivalence
\begin{myeq}\label{eqdiagcwe}
\diag\os X~\simeq~\diag c(X)~=~X~.
\end{myeq}

Now consider the unit map
$$
\hpi^{v}\os X~\to~N^{h}\hpi^{h}\hpi^{v}\os X~=~N^{h}\dP X~.
$$
for the comonad \w[.]{N\hpi:\Ss\to\Ss}

This map is a diagonal $2$-equivalence. In fact,
\w{(N^{v}N^{h}\dP X)^{h}_{0}=(N^{h}\dP X)^{h}_{0}=N\hpi\Dec X} and
\w[.]{(N^{v}\hpi^{v}\os X)_{0}^{h}=\Dec X} Moreover,
the map \w{\Dec X\to N\hpi\Dec X} is a weak equivalence, hence in
particular a $2$-weak equivalence, since the Kan complex \w{\Dec X} is
simplicially homotopy equivalent to a constant simplicial set
(cf.\ \cite[I, \S 6]{GJarS}). Further, for each \w[,]{i>0} the map:
$$
(N^{v}\hpi^{v}\os X)_{i}^{h}~\to~(N^{h}\dP X)_{i}^{h}~=~
(N^{h}\hpi^{h}\hpi^{v}\os X)_{i}^{h}
$$
is a $1$-equivalence, proving the claim.

It follows from Proposition \ref{pndiag} that there is a $2$-equivalence
\begin{myeq}\label{eqdiagnp}
\diag N^{v}\hpi^{v}\os X~\simeq~\dN\dP X.
\end{myeq}
\noindent Thus \wref[,]{eqdiagwe} \wref[,]{eqdiagcwe} and
\wref{eqdiagnp} imply that there is a $2$-equivalence
\begin{myeq}\label{eqnpo}
X~\simeq~\dN\dP X~.
\end{myeq}
Since, by hypothesis, $X$ is in \w[,]{\St} and by Proposition \ref{pfour}
\w{\dN\hP\os X} is in \w[,]{\St} it follows that
\wref{eqnpo} is a weak equivalence.

Finally, given \w[,]{\Gss\in\TGw} by \wref{eqdiagnp} there is a weak
equivalence
$$
\dN\G\simeq\dN\dP\dN\Gss
$$
This shows that \w{\Gss} and \w{\dP\dN\Gss} are weakly
equivalent. This completes the proof that the functors in \wref{eqeqloc} are
equivalences of categories.
\end{proof}

%
%
\section{Two-track categories}\label{cttrack}

We are now in a position to define the main subject of this paper:

\begin{defn}\label{dttrack}
A \emph{two-track category} with object set \w{\OO} is a
\ww{\TGwO}-category, where the enrichment is with respect to the
cartesian monoidal structure. We denote the category
of all two-track categories by \w[.]{\TTra}
\end{defn}

\begin{prop}\label{pfive}
The functor \w{\dP:\SOC\to\TTraO} associates to a fibrant
\ww{\SO}-category \w{\Xd} a two-track category \w{\G=\dP\Xd} with
object set $\OO$.
\end{prop}

\begin{proof}
\quad From its definition, it is straightforward that \w{\os} preserves
products. Since the fundamental groupoid functor preserves products,
and products in functor categories are computed pointwise, it follows
from \wref{eqhvp} that \w{\dP=\hP\circ\os:\Ss\to\TGpd} preserves
products. It thus extends to \w{\SOC} (applied to each mapping space).
\end{proof}

Similarly, the functor \w{\dN=\diag N^{v}N^{h}:\TGpd\to\Ss} preserves
products and therefore induces a functor on two-track categories
which we also denote by \w[.]{\dN:\TTraO\to\PSOC{2}}
We say that a morphism $f$ in \w{\TTraO} is a weak equivalence if
\w{\dN f} is a weak equivalence of \ww{\SO}-categories,
and we denote by \w{\TTraO\bsim} and \w[,]{\PSOC{2}\bsim}
respectively the localizations with respect to weak equivalences (so
that \w{\PSOC{2}\bsim} is actually the full subcategory
\w{\ho\PSOC{2}} of \w{\ho\SOC} with respect to the model category
structure described in \S \ref{ssoc}). We then deduce from Theorem
\ref{tmodeltt}:

%
%
\begin{cor}\label{cmodeltt}
The functors
$$
\dP:\PSOC{2}\to\TTraO\hspace*{7mm}\text{and}\hspace*{7mm}\dN:\TTraO\to\PSOC{2}
$$
induce equivalence of categories between \w{\PSOC{2}\bsim} and
\w[.]{\TTraO\bsim}
\end{cor}

\begin{mysubsection}{The homotopy track category of a two-track category}
\label{shtc}
Because a two-track category $\G$ is enriched in two-typical double
groupoids, which are in particular symmetric, we can apply \wref{eqwkglgpd}
to each mapping object to obtain a concise description of $\G$ in the form
\mydiagram[\label{eqtwotrack}]{
\G_{[1]}\ar@<0.7ex>[rr]^{d_{0}} \ar@<-.7ex>[rr]_{d_{1}}  &&
\G_{[0]}\ar@/_1.7pc/[ll]_{s}
}
\noindent in which \w{\G_{[1]}=(\G_{2}\toto\G_{1})} and
\w{\G_{[0]}=(\G_{1}\toto\G_{0})} are track categories, as in
\wref[.]{eqwkglgp}
We call a morphism $\xi$ of \w{\G_{1}} a \emph{$1$-track}, and write
\w[,]{\xi:f\Ra g} where \w{f=d_{0}\xi} and \w[.]{g=d_{1}\xi} Similarly,
a morphism $\alpha$ of \w{\G_{2}} is called a \emph{$2$-track}, and we write
\w[.]{\alpha:d_{0}\alpha\Ra d_{1}\alpha}

The additional data mentioned in Remark \ref{rpost} is available also
here: thus the \emph{homotopy track category} of a two-track category $\G$,
written \w{\hPi\G} or \w[,]{\ho\G} is obtained
by taking the coequalizer of the maps of track categories
\w[,]{\G_{[1]}\toto\G_{[0]}} or equivalently, by applying the fundamental
groupoid functor \w{\hpi} to each two-typical mapping object
\w{\G(a,b)} \wb[.]{a,b\in\G_{0}} This track category \w{\D:=\hPi\G}
itself has an associated homotopy category \w{\ho\D=\Pi_{0}\D}
(see \S \ref{dtrack}), which we denote simply by \w[.]{\Pi_{0}\G}

Moreover, we have a certain abelian track category \w{\Pi_{2}\G} over
\w[,]{\hPi\G} which will be described in the next section (see \S
\ref{egnatmod}), which together with \w{\hPi\G} (and \w[)]{\Pi_{0}\G}
can be thought of as the Postnikov system for $\G$. We summarize this
situation by the following diagram:
\mydiagram[\label{eqtwotrackext}]{
\Pi_{2}\G \ar[rr] && \G_{[1]}\ar@<0.7ex>[rr]^{d_{0}} \ar@<-.7ex>[rr]_{d_{1}}  &&
\G_{[0]}\ar[rr]&& \hPi\G
}
\noindent In Section \ref{cttetc} we will further analyze the
structure of \wref{eqtwotrackext} as a ``two-track extension'' of
\w{\hPi\G} by \w[.]{\Pi_{2}\G}
\end{mysubsection}

\begin{mysubsection}{A Gray category model of two-track categories}\label{sgcmt}
We wish to point out some connections with other higher-categorical structures.
These are not needed for the subsequent sections.

There is a functor
\w{L:\TGw\to\Tam} defined as follows: in \w[,]{\N^{h}\Gss} replace the groupoid
\w{G_{0\ast}} by the discrete groupoid \w[,]{G_{0\ast}^{d}} with
\w{d_{i}:G_{1\ast}\to G_{0\ast}} \wb{i=0,1} by \w[,]{\gamma d_{i}}  and
\w{s_{0}:G_{0\ast}\to G_{1\ast}} by \w{s_{0}\sigma} where
\w{\sigma:G_{0\ast}^{d}\to G_{0\ast}} is the natural section.
The verification that the resulting simplicial
groupoid \w{L(\Gss)} is a Tamsamani weak $2$-groupoid is as in the proof of
Proposition \ref{pfour}. In addition, there is a natural weak equivalence
\w{L(\Gss)\to T(\Gss)} (see proof of Proposition \ref{pfour}),
so \w{L(\Gss)} is weakly equivalent to \w[.]{N^{h}(\Gss)} This functor $L$
clearly preserves products.

Recall from \cite{TamsN,LPaolT} that there is a product-preserving functor
\w{M:\Tam\to\BGpd} to the category of bigroupoids. Furthermore, there is a
strictification functor \w{\Str:\BGpd\to\hy{2}{\Gpd}} to the category of strict
$2$-groupoids which is monoidal with respect to the cartesian product in
\w{\BGpd} and the Gray tensor product in \w{\hy{2}{\Gpd}} (see \cite{GPStreeA}).

All three functors $L$, $M$, and \w{\Str} preserve homotopy types (of the
classifying spaces). The composite
\w{\St\circ M\circ L:\TGw\to\hy{2}{\Gpd}} is also monoidal,
and thus extends to a functor \w[,]{K:\TTraO\to\hy{(\hy{2}{\Gpd},\OO)}{\Cat}}
where the target is the full subcategory of Gray categories whose $2$- and
$3$-cells are invertible.

The Street nerve functor \w{\N:\hy{2}{\Gpd}\to\Ss} (cf.\
\cite{StreeA}, and see \S \ref{sntc} below) is monoidal with respect to
the Gray tensor product in \w{\hy{2}{\Gpd}} and the cartesian product in $\Ss$
(see \cite{VeriCS}), and is weakly equivalent to the diagonal nerve. Therefore,
it extends to a functor \w[,]{\N:\hy{(\hy{2}{\Gpd},\OO)}{\Cat}\to\SOC} and we
have:
\end{mysubsection}

%
%
\begin{prop}\label{pgray}
The functors
$$
\St\circ M\circ
L:\PSOC{2}\to\hy{(\hy{2}{\Gpd},\OO)}{\Cat}\hspace*{3mm}\text{and}\hspace*{3mm}
\N:\hy{(\hy{2}{\Gpd},\OO)}{\Cat}\to\PSOC{2}
$$
induce equivalence of categories between \w{\PSOC{2}\bsim} and
\w[.]{\hy{(\hy{2}{\Gpd},\OO)}{\Cat}\bsim}
\end{prop}

Here weak equivalences in \w{\hy{(\hy{2}{\Gpd},\OO)}{\Cat}} are
defined by means of $\N$.
%
%
\section{Coefficient systems on track categories}
\label{ccstc}

In order to define a cohomology theory for track categories, we first
have to describe the possible coefficient systems. These have already
been identified, in Quillen's approach, as the abelian group objects
in the appropriate over-category (cf.\ \cite[\S 2]{QuiC}):

\begin{defn}\label{dmodule}
A \emph{module} over a track category
\w{\D=(\D_{1}\toto\D_{0})} is an abelian group object
\w{\M=(\M_{1}\toto\M_{0})} in the over-category \w[.]{\GDC/\D}
\end{defn}

In order to make this more explicit, note that by definition a module
$\M$ over $\D$ comes equipped with a map of track categories (in fact,
of \ww{\GD}-categories) \w[,]{p:\M\to\D} which has a section
\w{s:\D\to\M} (with \w[).]{p\circ s =\Id}

Applying the ``semi-discretization'' functor \w{(-)\ud}
of \S \ref{dtrack} yields a map of semi-discrete track categories
\w[,]{p\ud:\M\ud\to\D\ud} and we denote by
\w{\K:=\Ker(p\ud)} the kernel of this map (as a disjoint union of
homomorphisms of groups), so that \w[.]{\K\in\GdDC} In fact, since
\w[,]{\M\in(\GDC/\D)_{\ab}} we see that \w[.]{\K\in(\GdDC)_{\ab}}

\begin{remark}\label{rhorcomp}
Note that in any track category $\D$, the categorical (horizontal)
composition \w{\ch:\D_{i}(a,b)\times\D_{i}(b,c)\to\D_{i}(a,c)} \wb{i=0,1}
is determined by the interchange law:
\begin{myeq}\label{eqinterch}
\xi\ch\zeta~=~[(f_{0})_{\ast}\zeta]\cv[g_{1}^{\ast}\xi]~=~
[g_{0}^{\ast}\xi]\cv[(f_{1})_{\ast}\zeta]
\end{myeq}
\noindent for tracks \w{\zeta:g_{0}\Ra g_{1}:a\to b} and
\w[.]{\xi:f_{0}\Ra f_{1}:b\to c}
Here \w[,]{(f_{0})_{\ast}\zeta=s_{0}(f_{0})\ch\zeta} where the
degeneracy \w{s_{0}:\D_{0}\to\D_{1}} embeds \w{\D_{0}} in \w{\D_{1}}
by mapping \w{f:a\to b} to the identity track \w{\Id:f\Ra f} in
\w[.]{\D_{1}(a,b)}
\end{remark}

\begin{defn}\label{dnatsyst}
A \emph{natural system} (in abelian groups) on a track category
\w{\D=(\D_{1}\toto\D_{0})} is a collection $\K$ of abelian groups \w{K_{f}}
indexed by \w[,]{f\in\D_{0}(a,b)} equipped with:
\begin{enumerate}
\renewcommand{\labelenumi}{(\alph{enumi})\ }
\item For every composable sequence
\w{a\xra{g}b\xra{f}c} in \w[,]{\D_{0}} \emph{pre- and
  post-composition} homomorphisms \w{f_{\ast}:K_{g}\to K_{fg}}  and
\w[,]{g^{\ast}:K_{f}\to K_{fg}} with \w{(fg)^{\ast}=g^{\ast}f^{\ast}}
and \w[.]{(fg)_{\ast}=f_{\ast}g_{\ast}}
\item A homomorphism \w{\xi_{\ast}:K_{f_{0}}\to K_{f_{1}}} for every
\w[,]{\xi\in\D_{1}(f_{0},f_{1})} such that for any
\w{\zeta\in\D_{1}(f_{1},f_{2})} we have
\w[.]{(\xi\cv\zeta)_{\ast}=\zeta_{\ast}\,\xi_{\ast}}

In particular, this reduces to an action of the group \w{\D_{1}(f,f)} on
\w{K_{f}} when \w[.]{f=f_{0}=f_{1}}
\item The two structures commute \wh that is, if we use
  \wref{eqinterch} to define \w{\ch:K_{f}\times K_{g}\to K_{fg}} for
  any \w{a\xra{g}b\xra{f}c} in \w[,]{\D_{0}} then:
\begin{myeq}\label{eqhorver}
(\xi_{\ast}\alpha)\ch(\zeta_{\ast}\beta)~=~
(\xi\ch\zeta)_{\ast}(\alpha\ch\beta)
\end{myeq}
\noindent for \w{\zeta:g_{0}\Ra g_{1}} in \w[,]{\D_{1}(a,b)}
\w{\xi:f_{0}\Ra f_{1}} in \w[,]{\D_{1}(b,c)} \w[,]{\alpha\in K_{f_{0}}}
  and \w[.]{\beta\in K_{g_{0}}}
\end{enumerate}
\end{defn}

%
%
\begin{prop}\label{pnatmod}
For a fixed track category \w[,]{\D=(\D_{1}\toto\D_{0})} there is a
bijective correspondence between modules $\M$ over $\D$ and natural systems
$\K$ on $\D$ (up to isomorphism), defined by
\w[,]{\M\mapsto\Ker(p\ud)} where \w{p\ud} is as in \S \ref{dmodule}.
\end{prop}

\begin{proof}
Since \w{\K:=\Ker(p\ud)} is an abelian group object in
\w[,]{\GdDC} for every \w[,]{a,b\in\OO:=\Obj(\D_{0})} \w{\K(a,b)} is a
abelian semi-discrete groupoid, so that it is a disjoint union of
abelian groups \w[,]{K_{f}} one for each \w[.]{f\in\D_{0}(a,b)}
Moreover, the maps of semi-discrete track categories
\w{K\to\M\ud\xra{p\ud}\D\ud} and \w{s\ud} induce a split exact
sequence of groups
\mydiagram[\label{eqsplit}]{
0\ar[r] & K_{f}\ar[r]^<<<<<{i_{f}} & \M_{1}(f,f)\ar[r]^{p_{f}}~
& \D_{1}(f,f)\ar[r] \ar@/_1.5pc/_{s}[l] & 1
}
\noindent for each \w[.]{f\in\D_{1}(a,b)} Thus if we write
\w{M_{f}:=\M_{1}(f,f)} and \w{D_{f}:=\D_{1}(f,f)} for the two
automorphism groups, we have a semi-direct product of groups:
\w[.]{M_{f}\cong K_{f}\ltimes D_{f}}

The horizontal composition in \w{\M_{1}} induces an (associative)
composition in $\K$, which is completely determined by the homomorphisms
\w{\ch:K_{f}\times K_{g}\to K_{gf}} \wh that is, by
\w{f_{\ast}} and \w[.]{g{\ast}}

If we fix \w[,]{a,b\in\OO} then because \w{G:=\M(a,b)} is a groupoid,
any \w{\psi\in G(f_{0},f_{1})} induces an isomorphism
\w{\psi_{\ast}:G(f_{0},f_{0})\cong G(f_{0},f_{1})}
for every \w[.]{f_{0},f_{1}\in\D_{0}(a,b)=\M_{0}(a,b)} In particular,
this holds for \w[,]{\psi=s(\xi)} for any
\w[.]{\xi\in \D_{1}(f_{0},f_{1})}

Combining this with \wref[,]{eqsplit} we see that
\begin{myeq}\label{eqdisun}
\sigma:~\coprod_{\D_{1}(f_{0},f_{1})}~K_{f_{0}}~\cong~G(f_{0},f_{1})
\end{myeq}
\noindent (as sets), where the isomorphism $\sigma$ maps
\w{(\xi,\alpha)} in \w{\D_{1}(f_{0},f_{1})\times K_{f_{0}}} to
\w{s(\xi)_{\ast}(i_{f_{0}}(\alpha))} in \w[,]{\M_{1}(f_{0},f_{1})}
which we shall denote simply by \w[.]{\xi_{\ast}(\alpha)}

Because \w{p\rest{G}} is a map of groupoids, the vertical composition
in $G$ respects the splitting \wref{eqdisun} in the sense that
$$
(\xi,\alpha)\cv(\xi',\alpha')=(\xi\cv\xi',\alpha+\xi_{\ast}^{-1}(\alpha'))
$$
\noindent for \w{f_{0}\xRa{\xi}f_{1}\xRa{\xi'}f_{2}} in
\w[,]{\D_{1}(a,b)} \w[,]{\alpha\in K_{f_{0}}}
and \w[.]{\alpha'\in K_{f_{1}}}

Finally, because \w{i:\K\to\M} and \w{s:\D\to\M} are maps of track
categories (in \w[),]{\GDC} we see that $\sigma$ of \wref{eqdisun}
respects the horizontal compostions \w[,]{\ch} so \wref{eqhorver}
holds. Moreover, it is clear from this description that one can
reconstruct the module $\M$ from the natural system $\K$.
\end{proof}

\begin{example}\label{egnatmod}
The main example of a natural system is obtained as follows:

Let \w{\Xd} be an \ww{\SO}-category \wh or equivalently, a simplicial
category in \w[,]{s\Cat} with fixed object sset $\OO$, and let
\w{\D:=\hpi\Xd} the track category obtained by applying the
fundamental groupoid functor to each simplicial mapping space
\w{\Xd(a,b)} \wb{a,b\in\OO} of \w{\Xd} (assuming these are Kan
complexes), so \w[.]{\D_{0}=X_{0}} Fix some \w[.]{n\geq 2}

For each \w[,]{f\in\D_{0}(a,b)} let \w[.]{K_{f}:=\pi_{n}(\Xd(a,b);f)}
Since both \w{\hpi} and \w{\pi_{n}} commute with products in $\Ss$, the
composition structure maps \w{\circ:\Xd(a,b)\times \Xd(b,c)\to\Xd(a,c)}
of the simplicially enriched category \w{\Xd} induce
\w[,]{\ch:K_{f}\times K_{g}\to K_{gf}} as well as
$$
\hpi(\Xd(a,b);f_{0},f_{1})\times \hpi(\Xd(b,c);g_{0},g_{1})~\to~
\hpi(\Xd(a,c);g_{0}\circ f_{0},g_{1}\circ f_{1})~.
$$
The action of the fundamental groupoid on the higher homotopy groups
defines the isomorphism \w{\xi_{\ast}:K_{f_{0}}\to K_{f_{1}}} for each
\w[,]{\xi\in\hpi(\Xd(a,b);f_{0},f_{1})} and satisfies
\wref{eqhorver} by naturality of this action (see \cite[\S 1.13]{BJTurnR}).
\end{example}

Yet a third way of looking at modules over track categories is the
following, in the spirit of \cite[\S 1]{BWirC}:

\begin{defn}\label{dfaccat}
For any track category $\D$, \w[,]{\Fac\D} the \emph{category of factorizations}
of $\D$  having as objects the maps of \w[,]{\D_{0}} and as morphisms
\w{(h,k,\xi):f\to g} ``homotopy commuting'' squares:
\mydiagram[\label{eqfaccat}]{
a \ar[rr]^{f} & \ar@{=>}[d]_{\xi} &c \ar[d]^{k}\\
b \ar[u]_{h} \ar[rr]_{g} & & d
}
\noindent so that \w[.]{\xi:\ell\circ f\circ k\Ra g}
The composition defined by concatenation of squares:
\mydiagram[\label{eqcomptr}]{
a \ar[rr]^{f} & \ar@{=>}[d]_{\xi} &a' \ar[d]^{\ell}\\
b \ar[u]_{k} \ar[rr]_<<<<<<{g} & \ar@{=>}[d]_{\zeta} & b' \ar[d]^{n}\\
c \ar[u]_{m} \ar[rr]_{h} & & c'  \\
}
\noindent so \w{(k,\ell,\xi):f\to g} and \w{(m,n,\zeta):g\to h} compose to
\w[.]{(km,n\ell,m^{\ast}n_{\ast}\xi\oplus\zeta:f\to h)}
Note that \w{\Fac\D} is the Grothendieck construction on the functor
\w{\D_{0}\op\times\D_{0}\to\Cat} which sends \w{(a,b)} to the
groupoid \w[.]{\D(a,b)}

A natural system on $\D$ is then just a functor \w[.]{\K:\Fac\D\to\Abgp}
More generally, if $\C$ is any category, a \emph{natural system in $\C$}
on $\D$ is a functor \w[.]{\K:\Fac\D\to\C} Such a $\K$ assigns to each
\w{f:a\to b} in \w{\D_{0}} an object \w[,]{\K_{f}\in\C} and to each
diagram \wref{eqfaccat} a morphism
\w{\star_{(k,h,\xi)}:\K_{f}\to\K_{g}} in $\C$, where:
\begin{myeq}\label{eqtriop}
\star_{(k,h,\xi)}(\alpha)~=~\xi_{\ast}(k_{\ast}h^{\ast}(\alpha))
~=~\xi_{\ast}[s_{0}(h)\ch\alpha\ch s_{0}(k)]~,
\end{myeq}
\noindent (cf.\ \wref[).]{eqinterch} In other words, the operation
\w{\star_{(k,h,\xi)}} is composed of three simpler operations:
\w[,]{\xi_{\ast}} pre- and post-composition (as we can see by letting
$k$, $n$, and either $\xi$ or $\zeta$ in \wref{eqcomptr} be identity
maps or tracks).

The category of natural systems in $\C$ on $\D$ (with
natural transformations as morphisms) will be denoted by
\w[.]{\NS_{\D}(\C)=\C^{\Fac\D}}
\end{defn}

\begin{mysubsection}{Free natural systems}\label{sfreenat}
Note that there is a forgetful functor
\w[,]{U:\NS_{D}(\Set)\to\Obj(\Set)^{\Arr\D_{0}}} which assigns to a natural
system \w{\K:\Fac\D\to\Set} the corresponding set function on objects
\w[.]{\Obj(\K):\Obj(\Fac\D)=\Arr\D_{0}\to\Obj(\Set)} This has a left
adjoint \w[,]{\F:\Obj(\Set)^{\Arr\D_{0}}\to\NS_{D}(\Set)} constructed
as follows: given a function \w[,]{K:\Arr\D_{0}\to\Obj(\Set)} the
\emph{free natural system} \w{\F K:\Fac\D\to\Set} is defined for any
\w{g\in\Arr\D_{0}} by:
\begin{myeq}\label{eqfreens}
(\F K)_{g}~:=~ K(f)\times\Hom_{\Fac\D}(f,g)~,
\end{myeq}
\noindent where we identify \w[,]{(x,(h,k,\xi))} with
\w[,]{\xi_{\ast}(k_{\ast}h^{\ast}(x))} where \w{x\in K(f)} and
\w{(h,k,\xi):f\to g} are as in \wref[.]{eqfaccat}  The description of \w{\F K} on
morphisms is given by \wref[.]{eqcomptr}  Note that \w{K(g)} embeds in
\w{(\F K)_{g}} by \w[.]{x\mapsto (x,\Id)}
\end{mysubsection}

%
%
\section{Cohomology of track categories}
\label{cctc}

In order to define a Baues-Wirsching type cohomology theory for track
categories, we first consider:

\begin{mysubsection}{Nerves for track categories}\label{sntc}
Several concepts of nerves for higher categories have appeared in the
literature (see \cite{BergCN,BBFaroF,JWaltN}). We shall need the
following version, due in this form to Steet in \cite{StreeA} (but
see also \cite{CordD}):

If \w{\D=(\D_{1}\toto\D_{0})} is a track category, its \emph{nerve}
\w{\N\D} is the $3$-coskeletal simplicial set \w{\Nd} with:
\begin{enumerate}
\renewcommand{\labelenumi}{(\alph{enumi})\ }
\item \w{\sk{1}\Nd=\N\D_{0}} (the usual nerve of the category \w[);]{\D_{0}}
\item \w{N_{2}} has a $2$-simplex for every \w[,]{\xi\in\D_{1}(g\circ h,f)}
with faces:
\mydiagram[\label{eqtwosimp}]{
& 1 \ar[rd]^{g} \ar@{=>}[d]_{\xi}& \\
0 \ar[rr]_{f} \ar[ru]^{h} && 2
}
\item \w{N_{3}} has a $3$-simplex $\tau$ for every identity in \w{\D_{1}} of
  the form
\begin{myeq}\label{eqthreesimp}
f_{\ast}\eta\oplus\zeta=h^{\ast}\theta\oplus\xi
\end{myeq}
\noindent for a diagram:
\mydiagram[\label{eqtoda}]{
& & & \\
0  \ar[r]^{h} \ar@/^{2.7pc}/[rr]^>>>>>>>{m} \ar@/^{3.9pc}/[rrr]^{\ell}
\ar@/_{3.9pc}/[rrr]_{\ell} &
1 \ar[r]^{g} \ar@{=>}[u]^{\eta}
\ar@{=>}[d]_{\xi} \ar@/_{2.7pc}/[rr]_<<<<<<<<{k} &
2 \ar @{=>}[u]_{\zeta} \ar @{=>}[d]^{\theta} \ar[r]^{f} & 3\\
& & &
}
\noindent so that the (flattened) $3$-simplex $\tau$ is:
$$
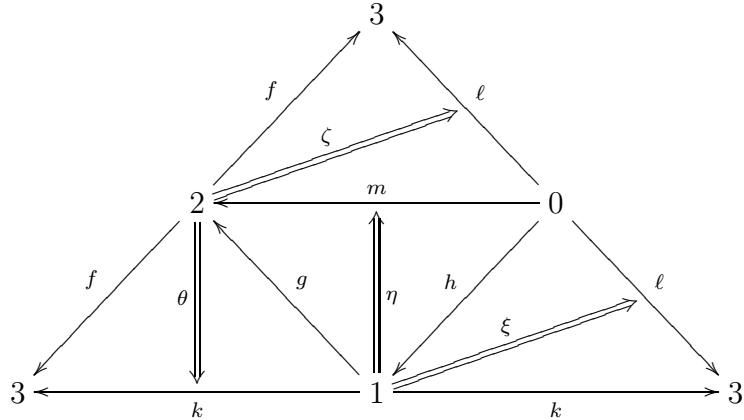
\xymatrix@R=25pt{
&&&&3&&&&\\
&&&&&&&&\\
&& 2 \ar[rruu]^{f} \ar[lldd]_{f}\ar@{=>}[rrru]^{\zeta} \ar@{=>}[dd]_{\theta}
&&&&
0 \ar[lluu]_{\ell} \ar[rrdd]^{\ell} \ar[lldd]_{h} \ar[llll]_{m} && \\
&&&&&&&&\\
3 &&&& 1 \ar[llll]^{k} \ar[lluu]_{g} \ar@{=>}[uu]_{\eta}
\ar@{=>}[rrru]^{\xi} \ar[rrrr]_{k} &&&& 3
}
$$
\noindent (with outer edges identified pairwise).
\end{enumerate}

By setting \w[,]{m=gh} \w[,]{k=fg} and \w{\ell=fgh} (with identity
tracks) we obtain \w{\zeta=f_{\ast}\eta} and \w[.]{\xi=h^{\ast}\theta}
Similarly, for \w[,]{h=\Id} \w[,]{m=g} and \w[,]{\eta=\Id:g\Ra g} we
  have \w[.]{\zeta\oplus\theta=\xi} Finally, since the (horizontal)
  composition of tracks \w{\eta:f\Ra f'} and \w{\theta:g\Ra g'}  satisfies
$$
\theta\ch\eta ~=~ g_{\ast}\eta\oplus (f')^{\ast}\theta
~=~ f^{\ast}\theta\oplus f'_{\ast}\eta~,
$$
\noindent one can recover all the structure of $\D$ from \w[.]{\Nd}
\end{mysubsection}

\begin{mysubsection}{Nerves and natural systems}\label{snns}
For any simplicial set \w{\Nd} define \w{\dm:N_{n}\to N_{1}} by
\w{\dm(\sigma):=d_{1}d_{2}\dotsb d_{n-1}\sigma\in N_{1}} (the edge
between the initial and final vertex of $\sigma$).
In particular, when \w{\Nd=\N\D} we may define
\w{N_{n}[f]:=\{\sigma\in N_{n}\D~:\ \dm(\sigma)=f\}} (for any arrow
$f$ in \w{\D_{0}} and \w[).]{n\geq 1} For \w{n=0} we set
\w{N_{n}[\Id_{a}]=\{a\}} and \w{N_{n}[f]=\emptyset} otherwise.
This defines a function \w[,]{\Arr\D_{0}\to\Obj(\Set)} and thus for
each \w{n\geq 0} we have a free natural system in sets on $\D$ denoted
by \w{\tN_{n}\D} (see \S \ref{sfreenat}).
\end{mysubsection}

\begin{remark}\label{rthreesimp}
Observe that for \w[,]{n=3} the collection \w{N_{3}\D} of all
$3$-simplices in the nerve of a track category $\D$ itself constitutes a
natural system (in sets) on $\D$, with \w{f\mapsto N_{3}[f]} on objects of
\w[.]{\Fac\D} We define the operation $\star$ of \wref{eqtriop} on a
$3$-simplex $\tau$ indexed by a diagram \wref{eqtoda} as follows:

Given a diagram \wref{eqfaccat} of the following special form:
$$
\xymatrix@R=25pt{
0 \ar[rr]^{\ell} & \ar@{=>}[d]_{\lambda} & 3 \ar[d]^{e}\\
0 \ar[u]^{\Id} \ar[rr]_{n} & & 4
}
$$
\noindent we define \w{\lambda_{\ast}e_{\ast}\tau} to be the
$3$-simplex indexed by
\mydiagram[\label{eqtodanew}]{
& & & & & & \\
0  \ar[rr]^{h} \ar@/^{2.7pc}/[rrrr]^>>>>>>>{m} \ar@/^{3.7pc}/[rrrrrr]^{n}
\ar@/_{3.7pc}/[rrrrrr]_{n} &&
1 \ar[rr]^{g} \ar@{=>}[u]^{\eta}
\ar@{=>}[d]_{e_{\ast}\xi\oplus\lambda}
\ar@/_{2.7pc}/[rrrr]_<<<<<<<<{ek} & &
2 \ar @{=>}[u]_{e_{\ast}\zeta\oplus\lambda} \ar[rr]^{ef} \ar
@{=>}[d]^{e_{\ast}\theta} & & 4~.\\
& & & & & &
}
\noindent The precomposition \w[,]{\mu_{\ast}i^{\ast}\tau} when
\wref{eqfaccat} has the form:
$$
\xymatrix@R=25pt{
0 \ar[rr]^{\ell} & \ar@{=>}[d]_{\mu} & 3 \ar[d]^{\Id}\\
-1 \ar[u]^{i} \ar[rr]_{n} & & 3,
}
$$
\noindent is defined analogously.

This allows us to think of a $3$-simplex $\tau$ indexed by the diagram
\wref{eqtoda} as encoding a morphism in \w{\Fac\D} from $g$ to $\ell$
(compare \S \ref{dfaccat}), with $\tau$ thus representing the operation
\begin{myeq}\label{eqactns}
\tau_{\ast}~:=~\eta_{\ast}h^{\ast}\zeta_{\ast}f_{\ast}
\end{myeq}
in any natural system \w[.]{\K:\Fac\D\to\C}

In particular, any $2$-simplex $\rho$ as in \wref{eqtwosimp} can be
thought of as a degenerate $3$-simplex, and thus yields two operations
for a natural system \w[,]{\K:\Fac\D\to\C} namely:
\begin{myeq}\label{eqacttwos}
\rho_{\ast}:=\xi_{\ast}g_{\ast}:\K_{h}\to\K_{f}\hsp\text{and}\hsp
\rho^{\ast}:=\xi_{\ast}h^{\ast}:\K_{g}\to\K_{f}~.
\end{myeq}
\end{remark}

\begin{mysubsection}{Baues-Wirsching type cohomology}\label{sbwc}
If $\D$ is a track category, the face and degeneracy maps of the nerve
\w{\Nd=\N\D} induce maps of natural systems as follows:

\begin{enumerate}
\renewcommand{\labelenumi}{(\alph{enumi})\ }
\item If \w{\phi=d_{i}:N_{n}\D\to N_{n-1}\D} \wb{0<i<n} or
\w{\phi=s_{j}:N_{n}\D\to N_{n+1}\D} \wb[,]{0\leq j\leq n}
we define \w{\tphi:\tN_{n}\D\to \tN_{n\pm 1}\D} to be \w[.]{\F\phi}
\item Given \w[,]{\sigma\in N_{n}\D} consider the $2$-simplex
  \w{\rho_{0}=d_{2}\dotsb d_{n-1}\sigma} of \w[,]{\N\D} and define
the map of natural systems \w{\td_{0}:\tN_{n}\D\to\tN_{n-1}\D} by setting
\w[,]{\td_{0}(\iota(\sigma)):=\rho_{0}^{\ast}(d_{0}\sigma)}
(see \wref[).]{eqacttwos} This extends to all of \w{\tN_{n}\D} by the
  adjointness of \S \ref{sfreenat}.
\item Similarly define \w{\td_{n}:\tN_{n}\D\to\tN_{n-1}\D} by setting
\w[.]{\td_{n}(\iota(\sigma)):=(d_{1}\dotsb d_{n-2}\sigma)_{\ast}(d_{n}\sigma)}
\end{enumerate}

This makes \w{\tNd:=(\tN_{n}\D)_{n=0}^{\infty}} into a simplicial
object in the category \w{\NS_{\D}(\Set)} of natural systems in \w{\Set}
on $\D$. Compare this to the description of the usual
Baues-Wirsching complex for a category $\C$ in terms of a two-sided
bar construction \w{\Bd(\G)} in \cite[\S 3]{BJTonkC}.

Now let $\K$ be a natural system (in \w[)]{\Abgp} on a track category
$\D$. We define a cosimplicial set \w{\Cu(\D;\K)} by setting
\w[,]{C^{n}(\D;\K):=\Hom_{\NS_{\D}(\Set)}(\tN_{n}(\D),U\K)} where
\w{U:\Abgp\to\Set} is the forgetful functor.  Since \w{U\K} is an
abelian group object in \w[,]{\NS_{\D}(\Set)} we see
that \w{\Cu(\D;\K)} is actually a cosimplicial abelian group (or
equivalently, a cochain complex). Its cohomotopy (i.e., the cohomology
of the corresponding cochain complex) is defined to be the
\emph{Baues-Wirsching cohomology} of $\D$ with coefficients in $\K$,
written \w{H^{n}_{\BW}(\D;\K):=\pi^{n}(\Cu(\D;\K))} (compare \cite[\S 1]{BWirC}
and \cite[\S 2]{BBlaC}).
\end{mysubsection}

\begin{remark}\label{rdisctrack}
Note that the embedding \w{\Set\hra\Gpd} defined by taking the
discrete groupoid on a set extends to an full and faithful functor
\w[,]{\OC\hra\TraO} so we can think of any small category $\E$ as a
(discrete) track category \w[.]{\D_{\E}} Since \w{\Fac\D_{\E}}
coincides with \w[,]{\Fac\E} a natural system $\K$ on \w{\D_{\E}} is
just a natural system on $\E$ (see \cite[\S 1]{BWirC}), and thus
\w{H^{n}_{\BW}(\D_{\E};\K)} is naturally isomorphic to
\w[,]{H^{n}(\E;\K)} the Baues-Wirsching cohomology of $\E$.

Furthermore, if the track category $\D$ is actually an internal
category in groups (equivalently, a crossed module),
\w{H^{\ast}_{\BW}(\D,\K)} may be identified with the cohomology of the
classifying space \w{B\D} (see \cite{EllH} and \cite{PaolC}).
\end{remark}

%
%
\section{Comparison with \ww{\SO}-cohomology}
\label{csoc}

Recall from \S \ref{snot} that an \ww{\SO}-category \w{\Xd} is a
category enriched in simplicial sets (cf.\ \S \ref{snot}) \wh or
equivalently, a simplicial object in  \w{\Cat} \wh with constant
object set $\OO$.

\begin{mysubsection}{$\SO$-categories}\label{ssoc}
In \cite[\S 1]{DKanS}, Dwyer and Kan define a model category structure
on \w[,]{\SOC} in which the fibrations and weak equivalence are
defined objectwise (that is, on each mapping space \w[).]{\Xd(a,b)} As
noted in \S \ref{egnatmod}, if \w{\Xd} is fibrant, then \w{\hpi\Xd} is
a track category, and \w{\pi_{n}\Xd} is a module over \w{\hpi\Xd} for
each \w[.]{n\geq 2} Moreover, a map \w{\Phi:\Md\to\Nd} in \w{\SOC} is a weak
equivalence if and only if it induces an isomorphism in \w{\pi_{n}} for
all \w[.]{n\geq 1}

Note that homotopy functors for simplicial sets which strictly preserve
products extend to \w[,]{\SOC} with the usual properties.
For example, given an \ww{\SO}-category \w[,]{\Xd} for each
\w{n\geq 1} we have a \ww{\PSO{n}}-category \w{\Yd:=\Po{n}\Xd} in which each
mapping space \w{\Yd(a,b)} is the $n$-the Postnikov section \w[.]{\Po{n}\Xd(a,b)}

Similarly if $\M$ is a module over a track category $\D$ (\S \ref{dmodule}),
applying the twisted Eilenberg-Mac~Lane functor
\w{\EM{\D}{-}{n}:=K(-,n)\rtimes\N\D} objectwise to $\M$ (see
\cite[\S 5]{BDGoeR}) yields the \emph{(twisted) Eilenberg-Mac~Lane}
\ww{\SO}-category \w{\EM{\D}{\M}{n}} in \w[,]{\SOC/\N\D} using the natural map
\w[.]{q:\EM{\D}{\M}{n}\to \Po{1}\EM{\D}{\M}{n}\simeq\N\D} This map $q$ is
equipped with a section \w[,]{s:\N\D\to\EM{\D}{\M}{n}} making
\w{\EM{\D}{\M}{n}} into an abelian group object in \w{\SOC/\N\D}
(see \cite{DKanS} or \cite[\S 3]{BBlaC}).

Finally, for each \w{n\geq 0} there is an $n$-th \emph{$k$-invariant square}
functor, which assigns to each \ww{\SO}-category \w{\Yd} a homotopy
pull-back square:
%
\begin{myeq}\label{eqone}
\xymatrix@R=25pt{\ar @{} [dr] |<<<{\framebox{\scriptsize{PB}}}
\Po{n+1}\Yd \ar[r]^{p\ui{n+1}} \ar[d] & \
\Po{n}\Yd \ar[d]^{k_{n}}\\ B\D \ar[r] & \EM{\D}{\M}{n+2}
}
\end{myeq}
\noindent  (over \w[),]{B\D} where \w{\D:=\hpi\Yd} and $\M$ is the
module \w{\pi_{n+1}\Yd} over $\D$.
The map \w{k_{n}:\Po{n}Y\to\EM{\D}{\M}{n+2}} is called the $n$-th
(functorial) $k$-\emph{invariant} for \w[.]{\Yd} Compare
\cite[Proposition 6.4]{BDGoeR}.
\end{mysubsection}

\begin{remark}\label{rfree}
There is a forgetful functor \w{U:\Cat\to\DG} to the category of directed
graphs, whose left adjoint \w{F:\DG\to\Cat} is the free category functor
(both $U$ and $F$ are the identity on objects). This pair of adjoint
functors defines a comonad \w[,]{\F:=FU:\Cat\to\Cat} and
thus an augmented simplicial category \w{\Ed\to\Cat} with
\w[.]{\E_{n}:=\F^{n+1}\C} If \w[,]{\C\in\OC} then \w[.]{\Ed\in\SOC}
A simplicial category \w{\Ed\in s\OC\cong\SOC} is \emph{free} if each category
\w[,]{E_{n}} and each degeneracy functor \w[,]{s_{j}:E_{n}\to E_{n+1}} is in
the image of the functor \w[,]{\Fs:\OC\to\OC} up to isomorphism.
\end{remark}

\begin{defn}\label{dsocoh}
Given a track category $\D$, a module $\M$, and an \ww{\SO}-category
\w{\Xd} equipped with a map \w[,]{\Xd\to\N\D} Dwyer, Kan, and Smith
define the $n$-th \ww{\SO}-\emph{cohomology group}
of \w{\Xd} \emph{with coefficients in} $\M$ to be
$$
H^{n}_{\SO}(\Xd/\N\D;\M)~:=~[\Xd,\EM{\D}{\M}{n}]_{\SOC/\N\D}~=~
\pi_{0}\map_{\N\D}\,(\Xd,\EM{\D}{\M}{n})~.
$$
When \w{\Xd\to\N\D} is a weak equivalence, we abbreviate
\w{H^{n}_{\SO}(\Xd/\N\D;\M)} to \w[.]{H^{n}_{\SO}(\D;\M)}
\end{defn}

Another model for topologically enriched categories is provided by
\w[,]{\COC} where $\C$ denotes the category of cubical sets:

\begin{defn}\label{dcube}
Let $\Box$ denote the \emph{Box category},  whose objects are the
abstract cubes \w{\{\Ic{n}\}_{n=0}^{\infty}} (where \w{\Ic{}:=\{0,1\}}
and \w{\Ic{0}} is a single point). The morphisms of $\Box$ are
generated by the inclusions \w{d^{i}_{\varepsilon}:\Ic{n-1}\to\Ic{n}}
and projections \w{s^{i}:\Ic{n}\to\Ic{n-1}} for \w{1\leq i\leq n}
and \w[,]{\varepsilon\in\{0,1\}}
A contravariant functor \w{K:\Box\op\to\Set} is called a
\emph{cubical set}.

We write \w{K_{n}} for the set \w{K(\Ic{n})} of \emph{$n$-cubes} of $K$.
The collection of all these, for \w[,]{n\geq 0} form a category
\w{\C_{K}} (with inclusions as morphisms). The two maps
\w{d_{i}^{\varepsilon}:K_{n}\to K_{n-1}} \wb{\varepsilon\in\{0,1\}}
induced by \w{d^{i}_{\varepsilon}} are called $i$-th \emph{face maps}
of $K$, and the map \w{s_{j}:K_{n}\to K_{n+1}} (induced by \w[)]{s^{j}} is
called the $j$-th \emph{degeneracy}. The sub-cubical set of $K$ generated
by \w{K_{0}\dotsc,K_{n}} under the degeneracies is called the $n$-th
\emph{cubical skeleton} of $K$, written \w[.]{\skc{n}K}
\end{defn}

\begin{defn}\label{dadjcs}
If $K$ and $L$ are two cubical sets, their \emph{cubical tensor}
\w{K\otimes L\in\C} is defined:
$$
K\otimes L~:=~\colim_{I^{j}\in\C_{K},~I^{k}\in\C_{L}}~I^{j+k}~.
$$
This defines a symmetric monoidal structure on cubical sets.

Cubical sets are related to simplicial sets by a pair of adjoint functors
\begin{myeq}\label{eqadj}
\C\adj{T}{\Sb}\Ss~,
\end{myeq}
\noindent where the \emph{triangulation} functor $T$ is defined by:
\w[,]{TK:=\colim_{I^{n}\in\C_{K}}~\Delta[1]^{n}} and the \emph{cubical singular}
functor \w{\Sb} is defined by \w[.]{(\Sb X)(I^{n}):=\Hom_{\Ss}(T I^{n},X)}
These induce equivalences of the corresponding homotopy
categories. For further details, see \cite{CisiP}, or the surveys in
\cite{JardC,JardCH}.
The adjoint pair \wref{eqadj} prolong to functors between \w{\COC} and
\w[,]{\SOC} which also induce equivalences of homotopy categories
(cf.\ \cite{BJTurnH}).
\end{defn}

\begin{mysubsection}{Models for \ww{\PCO{1}}-categories}\label{scubtrack}
Any track category $\D$ is equivalent to one in which \w{\D_{0}} is
free on its homotopy category \w{\Pi_{0}\D} (see \cite[(A.2)]{BDreC}).
We use this to give an explicit description of a cofibrant and fibrant
$2$-coskeletal \ww{\CO}-category $W$ weakly equivalent to \w{\Sb\N\D}
directly in terms of $\D$, in the spirit of the Boardman-Vogt
``W-construction'' (see \cite[\S 3]{BJTurnH}):

\begin{enumerate}
\renewcommand{\labelenumi}{(\alph{enumi})\ }
\item The $0$-cubes of the cubical mapping space \w{W(a,b)} correspond to maps
  in the free category \w[,]{\D_{0}=F\ho\D} which in turn are described by
  composable sequences
\begin{myeq}\label{eqcompseq}
\fd=(a=a_{n+1}\xra{f_{n+1}}a_{n}\xra{f_{n}}a_{n-1}\dotsc
a_{i+1}\xra{f_{i+1}}a_{i}\xra{f_{i}}a_{i-1}\dotsc a_{1} \xra{f_{1}}a_{0}=b)
\end{myeq}
\noindent in \w[.]{\ho\D} The corresponding vertex is denoted by
\w[,]{\ocmp{\fd}:=f_{1}\otimes\dotsc\otimes f_{n}\otimes f_{n+1}}
using the monoidal structure in $\C$.
\item For any track
$$
\xymatrix@R=25pt{
& b \ar[rd]^{f} \ar@{=>}[d]_{\xi}& \\
a \ar[rr]_{h} \ar[ru]^{g} && c
}
$$
\noindent in $\D$, we have a $1$-cube \w{\xi:f\otimes g\Ra h} in \w{W(a,b)}
(where either $f$ or $g$ could be \w[).]{\Id}
\item For any diagram in $\D$ of the form \wref[,]{eqtoda} satisfying
\wref[,]{eqthreesimp} we have a $2$-cube in $W$
\mydiagram[\label{eqsquare}]{
f\otimes g\otimes h \ar[rr]^{f\otimes\eta} \ar[d]^{\theta\otimes h} &&
f\otimes m \ar[d]^{\zeta} \ar[d]^{\zeta} \\
k\otimes h \ar[rr]^{\xi} && \ell~.
}
\item In general, an $n$-cube \w{\Ic{\fd}} of \w{W(a,b)} is determined
  by the following data:
\begin{enumerate}
\renewcommand{\labelenumii}{$\bullet$\ }
\item A composable sequence \w{\fd} in \w{\D_{0}} of length \w{n+1} as
 in \wref[.]{eqcompseq}
\item For every \w[,]{1\leq i\leq j\leq n} a map \w{h^{[i,j]}:a_{j}\to a_{i}}
  in \w[,]{\D_{0}} where \w{h^{[i,i+1]}:=f_{i+1}} and
  \w[.]{h^{[i,i]}:=\Id_{a_{i}}}
\item For each partition of the form:
\begin{myeq}\label{eqpartition}
\alpha=(0=i_{0}<i_{1}<\dotsc<i_{k-1}< i_{k}=n+1)~,
\end{myeq}
\noindent we thus obtain a composable sequence
\begin{myeq}\label{eqnewcompseq}
\hd[\alpha]=(a_{n+1}=a_{i_{k}}\xra{h^{[i_{k-1},i_{k}]}}a_{i_{k-1}}
\xra{h^{[i_{k-2},i_{k-1}]}}a_{i_{k-2}}\dotsc
a_{i_{1}}\xra{h^{[i_{0},i_{1}]}}a_{i_{0}}=a_{0})~.
\end{myeq}
\noindent in \w{\D_{0}} of length $k$. The vertices of \w{\Ic{\fd}}
are indexed by \w[,]{\alpha\in A} with the corresponding vertex denoted by
\w[.]{\ocmp{\gd[\alpha]}} In particular, the vertex indexed by
\w{\ocmp{\fd}=d_{1}^{1}d_{2}^{1}\dotsc d_{n}^{1}\Ic{\fd}} is called the
\emph{initial vertex} of \w[,]{\Ic{\fd}} and the vertex indexed by the
single map  \w{h^{[0,n+1]}=d_{1}^{0}d_{2}^{0}\dotsc d_{n}^{0}\Ic{\fd}}
is called the \emph{terminal vertex} of \w[.]{\Ic{\fd}} By analogy
with \S \ref{rthreesimp} we denote \w{h^{[0,n+1]}} by \w[.]{\dm\Ic{\fd}}
\item For every \w[,]{1\leq i\leq j\leq k\leq n} a track
  \w{\xi^{[i,j,k]}:h^{[i,j]}\circ h^{[j,k]}\Ra h^{[i,k]}} in
  \w[,]{\D_{1}} which is the identity track if \w{i=j} or \w{j=k} (or both).
\item These tracks are \emph{compatible} in that for every
\w[,]{1\leq i\leq j\leq k\leq \ell\leq n} the diagram
\mydiagram[\label{eqsqtra}]{
h^{[i,j]}\circ h^{[j,k]}\circ h^{[k,\ell]}
\ar@{=>}[rrr]^{h^{[i,j]}_{\ast}\xi^{[j,k,\ell]}}
\ar@{=>}[d]_{(h^{[k,\ell]})^{\ast}\xi^{[i,j,k]}}  &&&
h^{[i,j]}\circ h^{[j,\ell]} \ar@{=>}[d]^{\xi^{[i,j,\ell]}} \\
h^{[i,k]}\circ h^{[k,\ell]}\ar@{=>}[rrr]_{\xi^{[i,k,\ell]}} &&& h^{[i,\ell]}
}
\noindent commutes in \w[.]{\D_{1}}
\item Given $\alpha$ as above and \w[,]{0<j<k} let \w{(\alpha,\hat{j})}
  be the partition obtained from $\alpha$ by omitting \w[.]{i_{j}}
There is an edge in \w{\Ic{\fd}} from \w{\ocmp{\gd[\alpha]}} to
\w[,]{\ocmp{\gd[\alpha,\hat{j}]}} indexed by
$$
h^{[i_{0},i_{1}]}\otimes\dotsc\otimes h^{[i_{j-2},i_{j-1}]}\otimes
\xi^{[i_{j-1},i_{j},i_{j+1}]}\otimes h^{[i_{j+1},i_{j+2}]}\otimes \dotsc
h^{[i_{k-1},i_{k}]}~.
$$
\item The $2$-faces in \w{\Ic{\fd}} are similarly determined by $\alpha$ as
  above and a choice of \emph{two} indices \w{0<j'<j''<k} to be
  omitted, and so on.
\end{enumerate}
\end{enumerate}
\end{mysubsection}

\begin{remark}\label{rcubsphere}
Let \w{S^{n}\in\C_{\ast}} be a pointed cubical
sphere (e.g., \w[),]{I^{n}/\partial I^{n}} which corepresents
\w{\pi_{n}} for pointed cubical sets. Given a fibrant cubical set $K$,
with \w{V:=\Po{1}K} a fibrant cubical model for \w[,]{N\hpi K} the
\emph{twisted sphere} \w{S^{n}\otimes V} corepresents the
\ww{\hpi K}-module \w{\{\pi_{n}(K,k)\}_{k\in K_{0}}} in \w[.]{\C/V}

More generally, if \w{X\in\COC} is fibrant and $W$ is a cofibrant
model for \w[,]{\Po{1}X} then for each \w{a,b\in\OO} the twisted
cubical $n$-sphere \w{S^{n}\otimes W(a,b)} generates under pre- and
post-composition a \ww{\CO}-category \w{S^{n}_{(a,b)}\otimes W} in \w[,]{\COC/W}
with \w[.]{\skc{n-1}(S^{n}_{(a,b)}\otimes W)=\skc{n-1}W} This again
corepresents \w[,]{\pi_{n}X(a,b)} so the various choices of
\w{a,b\in\OO} (for fixed $n$) together corepresent the natural system
\w{\pi_{n}X} of Example \ref{egnatmod}.
\end{remark}

\begin{mysubsection}{Cubical Eilenberg-Mac~Lane categories}\label{scubemcat}
Let $\D$ be a track category and $\K$ be a natural system on $\D$,
with $\M$ the corresponding module over $\D$ (cf.\ Proposition
\ref{pnatmod}), and let $W$ be the \ww{\CO}-model for $\D$ constructed
above. For each \w[,]{n\geq 2} we can use it to construct an explicit
fibrant \ww{\CO}-model $E$ for \w[,]{\EM{\D}{\M}{n}} as follows:

We start with \w[,]{\skc{n-1}E:=\skc{n-1}W} and let
\w{W_{n}[f]:=\{I\in W_{n}~:\ \dm I=f\}} for each arrow
\w{f\in\D_{0}} (compare \S \ref{rthreesimp}). The degenerate $n$-cubes
of $E$ are those of $W$, and the $n$-cubes of $E$ are defined by setting
\w[,]{E_{n}:=\{(I,\alpha)~:\ I\in W_{n},\ \alpha\in\K_{\dm(I)}\}}
so \w{E_{n}[f]=\K_{f}\times W_{n}[f]} for each \w[.]{f\in\D_{0}} 
The face maps are zero if \w[,]{\alpha\neq 0} 
\w[,]{d^{\varepsilon}_{i}(I,0):=d^{\varepsilon}_{i}I\in W_{n-1}} and the 
degeneracies are formal (i.e.,  we add symbols \w{s_{j}(I,\alpha)\in E_{n+1}} 
for each \w[).]{1\leq j\leq n} Note that \w{(I,\alpha)} is never
degenerate for \w[,]{\alpha\neq 0} even though $I$ itself may be
degenerate (i.e., some or all of the factors \w{f_{i}} of the sequence
\w{\fd} indexing $I$ may be identity maps) or decomposable. 

For each \w{J\in W_{n+1}} indexed by
\w[,]{\fd=(a_{n+2}\xra{f_{n+2}}a_{n+1}\dotsc a_{1}\xra{f_{1}}a_{0})}
consider a collection \w{\ba=\lra{\alpha_{i}^{0},\alpha_{i}^{1}}_{i=1}^{n+1}}
of elements in $\K$, with
\begin{myeq}\label{eqcycleone}
\alpha_{i}^{\varepsilon}\in\K_{\dm(d_{i}^{\varepsilon}J)}
\end{myeq}
\noindent and
\begin{myeq}\label{eqcycletwo}
\alpha_{i}^{1}=0 \text{~~for~~}2\leq i\leq n~.
\end{myeq}

If $\ba$ satisfies the \emph{cycle condition}
\begin{myeq}\label{eqcycle}
\xi^{[0,n+1,n+2]}_{\ast}f_{n+2}^{\ast}\alpha_{1}^{1}~+~
\sum_{i=2}^{n}~(-1)^{i}(\alpha_{i}^{1}-\alpha_{i}^{0})~
+(-1)^{n+1}~\xi^{[0,1,n+2]}_{\ast}(f_{0})_{\ast}\,\alpha_{n+1}^{1}~=~0~,
\end{myeq}
\noindent then we have a unique \ww{(n+1)}-cube \w{(J,\ba)} in
\w[,]{E_{n+1}} with
\w[.]{d_{i}^{\varepsilon}(J,\ba)=(d_{i}^{\varepsilon}J,a_{i}^{\varepsilon})}
We can think of \wref{eqcycle} as a ``matching face condition'' for the
collection $\ba$, as in a Kan complex (cf.\ \cite[I, \S 3]{GJarS}).

Finally, \w{E=\EM{\D}{\M}{n}} is \ww{(n+1)}-coskeletal.
Using Remark \ref{rcubsphere}, we see that \w{\pi_{n}E\cong\M} as
$\D$-modules, and \w{\pi_{n}E=0} for \w[.]{2\leq i\neq n}
\end{mysubsection}

\begin{example}\label{egncubes}
Note that \wref{eqcycle} succinctly encodes the condition that
\w{\pi_{n}E\cong\M} as a module over $\D$ (see \S \ref{egnatmod}): by
making suitable choices of $\ba$ (with only two nonzero entries
\w[),]{a_{i}^{\varepsilon}} we can ensure that all of the identities
of \S \ref{dnatsyst} are satisfied.
For example, let \w[,]{(I,\alpha)\in E_{n}} with $I$ indexed by
\w[:]{\fd=(a_{n+1}\xra{f_{n+1}}\dotsc\xra{f_{1}}a_{0})}

\begin{enumerate}
\renewcommand{\labelenumi}{(\alph{enumi})\ }
\item Post-composing \w{(I,\alpha)} with the $0$-cube of \w{W\subseteq E}
indexed by \w{g:a_{0}\to b} yields a ``formal'' composite
\w[,]{I':=g\otimes(I,\alpha)\in E_{n}} with
\w[,]{\dm I'=g\otimes\dm I=g\cdot h^{[0,n+1]}} because the (free)
category \w{\D_{0}} is identified with \w{W_{0}} by construction.
In order for \w{I'} to be identified up to homotopy with
\w[,]{(I',g_{\ast}\alpha,)} we have an \ww{(n+1)}-cube \w{(J,\ba)\in E_{n+1}}
such that \w[,]{d_{1}^{0}J=I'} \w[,]{d_{1}^{1}J=(g\otimes I,g_{\ast}\alpha)}
and \w{d_{i}^{\varepsilon}J=(s_{i}d_{i}^{\varepsilon}I,0)} is
degenerate for all \w{2\leq i\leq n+1} and \w[.]{\varepsilon=0,1}
\item One can similarly construct an \ww{(n+1)}-cube identifying
\w{(I\otimes h,h^{\ast}\alpha)} with \w[.]{h\otimes(I,\alpha)}
\item If \w{\psi:h^{[0,n+1]}\Ra k} is a track in $\D$, let
\w{I''\in W_{n}} be the $n$-cube obtained from $I$ by replacing its
terminal vertex \w{\dm I=h^{[0,n+1]}} by $k$, and post-composing all
edges ending in \w{h^{[0,n+1]}} with $\psi$.
Again we have an \ww{(n+1)}-cube \w{(J,\ba)\in E_{n+1}}
with \w{d_{1}^{0}(J,\ba)=(I,\alpha)} and
\w[.]{d_{1}^{1}(J,\ba)=(I'',\psi_{\ast}\alpha)}
All other faces are determined by the requirement that all edges
are degenerate (indexed by \w[),]{\Id}
except for \w[.]{d_{2}^{0}d_{2}^{0}\dotsc d_{n+1}^{0}(J,\ba)=\psi}

For instance, if \w[,]{n=2} $I$ is as in \wref[,]{eqsquare} and
\w{\psi:\ell\to k} is a track in $\D$, \w{(J,\ba)} will be:
\mydiagram[\label{eqncubes}]{
%
%
f\otimes g\otimes h
\ar@{}[drrrrr] |<<<<<<<<<<<<<<<<<<<<<<<<<<<<<<<{\framebox{\scriptsize{$\alpha$}}}
\ar[rrrr]^>>>>>>>>>>>>>>>{f\otimes\eta} \ar[rrd]^{\theta\otimes h}
\ar[dd]_{\Id} &&&& f\otimes m \ar[rrd]^{\zeta}
\ar[dd]_>>>>>>>>>>>>>>>{\Id}  && \\
%
%
&& k\otimes h \ar[rrrr]^<<<<<<<<<<<<{\xi}
\ar[dd]_>>>>>>>>>>>>>>>{\Id} &&&& \ell \ar[dd]_{\psi} \\
%
%
f\otimes g\otimes h\ar@{}[drrrrr]
|<<<<<<<<<<<<<<<<<<<<<<<<<<<<<<<{\framebox{\scriptsize{$\psi_{\ast}\alpha$}}}
\ar[rrrr]^>>>>>>>>>>>>>>>{f\otimes\eta} \ar[rrd]^{\theta\otimes h} &&&&
f\otimes m \ar[rrd]^{\psi_{\ast}\zeta} && \\
%
%
&& k\otimes h
\ar[rrrr]^{\psi_{\ast}\xi} &&&& k
}
\noindent with the top square being \w[,]{(\alpha,I)} and the bottom
\w[.]{(\psi_{\ast}\alpha,I'')}
\item Finally, a suitable choice of \w{(J,\ba)} ensures that
\w{\alpha_{1}+\alpha_{2}=\alpha_{3}} in \w[\vsm.]{\pi_{n}E}
\end{enumerate}
\end{example}

We are now in a position to prove the following analogue of
\cite[Theorem 3.9]{BBlaC}:

%
%
\begin{thm}\label{tbwso}
Let $\D$ be a track category, with \w{\Xd:=\N\D} the corresponding
\ww{\PSO{1}}-category, and let $\K$ be a natural system on $\D$, with
$\M$ the corresponding module over $\D$.
For each \w[,]{n\geq 1} the $n$-th Baues-Wirsching cohomology group
\w{H^{n}_{\BW}(\D;\K)} is then naturally isomorphic to the \ww{(n-1)}-st
\ww{\SO}-cohomology group \w[.]{H^{n-1}_{\SO}(\N\D;\M)}
\end{thm}

\begin{proof}
First note that the adjoint pair \wref{eqadj} induce a natural
isomorphism between the \ww{\SO}-cohomology group
\w{H^{n}_{\SO}(\N\D;\M)} and the \ww{\CO}-cohomology group
$$
H^{n}_{\CO}(W/\D;\M):=[W,\EM{\D}{\M}{n}]_{\COC/W}~,
$$
\noindent where $W$ is a cofibrant model for \w[.]{\Sb\N\D}

As in \S \ref{scubtrack}, we may assume that \w{\D_{0}} is free, and identify
\w{f\otimes g} in \w{W_{0}} with the composite \w{fg} in \w[.]{\D_{0}}
For each arrow \w{f\in\D_{0}} set \w{W_{n}[f]:=\{I\in W_{n}~:\ \dm I=f\}}
(compare \S \ref{rthreesimp}).
Given \w{\bllt\xra{g}\bllt\xra{f}\bllt\xra{h}\bllt} in \w[,]{\D_{0}}
we may use the cubical pre- or post-composition with $0$-vertices of
$W$ to define operations \w{g^{\ast}: W_{n}[f]\to W_{n}[fg]} and
\w[,]{h_{\ast}:W_{n}[f]\to W_{n}[hf]} respectively.
Furthermore, for each track \w{\psi:f\circ g\Ra h} and
\w{I=I'\otimes I''\in W_{n}[f\otimes g]}
with \w{I'\in W_{n}[f]} and \w[,]{I''\in W_{n}[g]} we obtain a new
$n$-cube \w{\psi_{\ast}I} by post-composing all tracks indexing edges
into the terminal vertex \w{f\otimes g} of $I$ with $\psi$, as in
Example \ref{egncubes} \wh see bottom square in Figure \wref[.]{eqncubes}
Under these operations, \w{W_{n}} constitutes a natural system in sets
on \w[.]{\D_{0}}

Note that a map \w{\phi:W\to\EM{\D}{\M}{n}} of \ww{\CO}-categories
over $W$ is determined by its values on the indecomposable $n$-cubes
in $W$, and $\phi$ must take \w[,]{I\in W_{n}} with \w{f:=\dm I} to an
$n$-cube \w{(I,\tilde{\phi}(I))\in E_{n}} with \w[.]{\tilde{\phi}(I)\in\K_{f}}
Moreover, any \ww{(n+1)}-cube \w{J\in W_{n+1}}
must map to \w{(J,\ba)} satisfying \wref[,]{eqcycle} so that the map
\w{\tilde{\phi}:W_{n}\to\K} described above is in fact a map of
natural systems over $\D$. Conversely, any such map $\tilde{\phi}$ in
\w{\NS_{\D}} satisfies the cocycle condition, and so uniquely extends to
\w[;]{W_{n+1}} it thus determines a map \w{\phi:W\to\EM{\D}{\M}{n}}
over $W$, since both $W$ and \w{\EM{\D}{\M}{n}} are \ww{(n+1)}-coskeletal.
Therefore, we can compute \w{[W,\EM{\D}{\M}{n}]_{\COC/W}} as the
cohomology of the cochain complex associated to the cocubical abelian group
\w[.]{\Cus:=\Hom_{\NS_{\D}}(W,\K)} \wh see \cite[\S 3]{BBlaC} and
compare \cite[\S 4]{BlaQ}.

Finally, the $n$-cubes in \w[,]{W_{n}[f]} indexed by composable
sequences \w{\fd} of length \w{n+1} in \w[,]{\D_{0}} with the
additional track data described in \S \ref{scubtrack}, are clearly in
one-to-one correspondence with the \ww{(n+1)}-simplices of \w[.]{\N\D}
By \wref[,]{eqcycletwo} the face maps \w{d^{1}_{i}} \wb{1<i<n+1}
are not relevant, so we can actually identify the cochain complex
associated to the cocubical abelian group
\w{\Cus:=\Hom_{\NS_{\D}}(W,\K)} with that associated to the
cosimplicial abelian group
\w{\Cu(\D;\K):=\Hom_{\NS_{\D}}(\tN(\D),\K)} used to define
\w{H^{n}_{\BW}(\D;\K)} (with a dimension shift).
\end{proof}

%
%
\section{Two-track extensions of track categories}
\label{cttetc}

Now let $\G$ be a two-track category.  As noted in \S \ref{shtc}, we can
associate to $\G$ its homotopy track category \w{\Pi_{0}\G=\ho\G}
by taking the coequalizer of the maps of track categories
\w{\G_{2}\toto\G_{1}} (in the notation of \wref[).]{eqtwotrack}
At the same time, one can associate with $\G$ a natural system
\w{\K:=\Pi_{2}\G} on $\D$, which may be identified with
\w{\pi_{2}\N\G} described in Example \ref{egnatmod}.

\begin{defn}\label{dbwcc}
A \emph{two-track extension} of track category $\D$ by a natural system
$\K$ on $\D$ is two-track category \w{\G=(\G_{[1]}\toto\G_{[0]})} such
that \w{\ho\G} is weakly equivalent to $\D$, and $\K$ is isomorphic to
\w{\Pi_{2}\G} under this equivalence. We use the following diagram to
describe this situation:
\mydiagram[\label{eqttrkext}]{
\K \ar[r] & \G_{[1]}\ar@<0.7ex>[r]^{d_{0}} \ar@<-.7ex>[r]_{d_{1}}  &
\G_{[0]}\ar[r] & \D
}
\noindent (compare \wref[).]{eqtwotrackext}
\end{defn}

As an immediate consequence of Theorem \ref{tbwso} we deduce
(compare \cite[(4.6)]{BDreC}):

%
%
\begin{cor}\label{cbwclass}
Equivalence classes of two-track extensions of a track category $\D$
by a natural system $\K$ are in one-to-one correspondence with
elements of \w[.]{H^{4}_{\BW}(\D;\K)}
\end{cor}

\begin{proof}
Let \w{\Yd:=\dN\G} be the \ww{\PSO{2}}-category corresponding to a
two-track extension $\G$ of $\D$ by $\K$, and let \w{\M=\pi_{2}\Yd} be
the module over \w{\hpi\Yd} corresponding to the natural system $\K$
under the identification of \w{\hpi\Yd} with $\D$.
By \cite[\S 3]{DKSmO}, a \ww{\PSO{2}}-category \w{\Yd} with first
Postnikov section \w{\Xd:=\Po{1}\Yd} is determined up to weak
equivalence by its first $k$-invariant  \w[.]{k_{1}}
By Corollary \ref{cmodeltt}, $\G$ is determined up to weak
equivalence by \w[,]{\Yd} and thus by
\w[.]{k_{1}\in H^{3}_{\SO}(\N\D;\M)}
By Theorem \ref{tbwso} this cohomology group may be identified with
\w[.]{H^{4}_{\BW}(\D;\M)}
\end{proof}

\begin{mysubsection}{The Baues-Wirsching type class}\label{sbwcl}
In \cite{BWirC} an explicit cohomology class in \w{H^{3}_{\BW}(\E;\M)}
was constructed classifying all \emph{linear} track extensions of the
category $\E$ by a natural system $\M$ on $\E$ \wh that is, track
categories $\D$ such that \w{\ho\D\cong\E} and \w[.]{\pi_{1}\D\cong\M}
(Not all track categories have this form, since in general we can only
expect \w{\pi_{1}\D} to be a natural system in \emph{groups}.)
We now do the same for all two-track extensions:

Let $\G$ be a two-track category with homotopy track category
\w[,]{\D=\Pi_{0}\G} and natural system \w{\K:=\Pi_{2}\G} on $\D$.
We choose a section \w{s:\D\to\G_{[1]}} for $\gamma$ as above \wh
that is, for each \w[,]{a,b\in\OO} we choose once and for
all a map \w{s(f)\in\G_{0}(a,b)} representing each homotopy class
\w[,]{f\in\D_{0}(a,b)} and a $2$-cell \w{s(\xi):s(f)\to s(f')} in
\w{\G_{1}(a,b)} for each track \w{\xi:f\Ra f'} in \w{\G_{1}(a,b)}
(with no compatibility assumptions).

A $4$-cocycle in \w{\Cu(\D;\K)} is given by a map of natural systems
\w[,]{\tilde{\phi}:\tN_{4}(\D)\to\K} which is determined in turn by
assigning to each $4$-simplex \w{\sigma\in\N_{4}\D} an element
\w[.]{\phi(\alpha)\in K_{\dm\sigma}} Now any $3$-simplex
\w{\tau\in\N_{3}\D} is determined by a diagram \wref{eqtoda} in $\D$
satisfying \wref[,]{eqthreesimp} which means that its lift
\mydiagram[\label{eqstoda}]{
& & & \\
0  \ar[r]^{s(h)} \ar@/^{2.7pc}/[rr]^>>>>>>>{m} \ar@/^{3.9pc}/[rrr]^{s(\ell)}
\ar@/_{3.9pc}/[rrr]_{s(\ell)} &
1 \ar[r]^{s(g)} \ar@{=>}[u]^{s(\eta)}
\ar@{=>}[d]_{s(\xi)} \ar@/_{2.7pc}/[rr]_<<<<<<<<{s(k)} &
2 \ar @{=>}[u]_{s(\zeta)} \ar @{=>}[d]^{s(\theta)} \ar[r]^{s(f)} & 3\\
& & &
}
\noindent has a $2$-track \w{s(\tau)\in\G_{2}} (see \S \ref{shtc}) with
\begin{myeq}\label{eqsthreesimp}
s(\tau):s(f)_{\ast}s(\eta)\oplus s(\zeta)~\Ra~s(h)^{\ast}s(\theta)\oplus s(\xi)
\end{myeq}
\noindent which represents an element \w[,]{\psi_{s}(\tau)\in K_{\ell}} by
definition of \w[.]{\K=\Pi_{2}\G} This extends to a map of natural
systems \w[.]{\tpsi_{s}:\tN_{3}(\D)\to\K}

Finally, we define \w{\tilde{\phi}:\tN_{4}(\D)\to\K} by setting
\begin{myeq}\label{eqfourcocycle}
\phi(\sigma)~:=~\sum_{i=0}^{4}~(-1)^{i}\tpsi_{s}(\td_{i}\sigma)
\text{~~~~~for~~~}\sigma\in N_{4}\D~.
\end{myeq}
\end{mysubsection}

%
%
\begin{lemma}\label{lcocycle}
\w{\tphi} is a $4$-cocycle in \w[.]{C^{4}(\D;\K)}
\end{lemma}

\begin{proof}
Given a $5$-simplex \w[,]{\alpha\in N_{5}\D} we see
$$
(\delta\tphi)(\alpha)~:=~
\sum_{j=0}^{5}~(-1)^{j}(\td^{j}\tphi)(\alpha)~=~
\sum_{j=0}^{5}~(-1)^{j}\tphi(\td_{j}\alpha)~=~
\sum_{j=0}^{5}~(-1)^{j}
\left(\sum_{i=0}^{4}~(-1)^{i}\tpsi_{s}(\td_{i}\td_{j}\alpha)\right)
$$
\noindent which vanishes by the usual simplicial identities.
\end{proof}

\begin{defn}\label{dbwclass}
We denote the cocycle defined above by
\w[,]{\tphi_{\G}\in C^{4}(\hpi\G;\Pi_{2}\G)} and the corresponding
cohomology class, called the \emph{Baues-Wirsching class} for $\G$,
by \w[.]{\chi_{\G}\in H^{4}(\hpi\G;\Pi_{2}\G)}
\end{defn}

%
%
\begin{thm}\label{tbwclass}
If \w{\Vd} is a \ww{\PSO{2}}-category and $\G$ is the associated
two-track category, the first $k$-invariant for \w{\Vd} corresponds to
the cohomology class \w{\chi_{\G}} defined above under the natural
isomorphism of Theorem \ref{tbwso}.
\end{thm}

\begin{proof}
Let \w{X=\Sb\Vd} be the \ww{\PCO{2}}-category corresponding to \w[,]{\Vd}
and consider the following square of the form \wref{eqone} in
\w[:]{\COC}
\mydiagram[\label{eqzeropo}]{
\ar @{}[drr] |>>>>>{\framebox{\scriptsize{PO}}}
X=\Po{2}X \ar@<2ex>[d]_{p\ui{2}}\ar[rr]^{i\ui{2}}
& & Y \ar[d]^{q} & W=B\D? \ar[l]_<<<<<{\varphi}^>>>>>{\simeq} \\
B\D=\Po{1}X \ar[rr]_{j} & & Z~, &
}
\noindent where \w{\D=\hpi X\cong\hpi\Vd} is the track category
associated to $X$.

By \cite[Proposition 6.4]{BDGoeR}, the homotopy pushout $Z$ in
\w{\COC} satisfies \w[.]{\Po{3}Z\simeq\EM{\D}{\pi_{2}X}{3}}
Thus if \w{r\ui{3}:Z\to \Po{3}Z} is the structure map of the Postnikov
tower, the first $k$-invariant for \w{\Vd} (or equivalently, for $X$)
is represented by the map \w{k_{1}:=r\ui{3}\circ q\sim r\ui{3}\circ j} in
\w[.]{[W,\EM{\D}{\pi_{2}X}{3}]_{\COC/W}}

We use the cubical version of Kan's original model for the Postnikov
system, so that \w{(\Po{k}X)_{n}} consists of \ww{\sim_{k}}-equivalence
classes of $n$-cubes in $X$, where \w[,]{I\sim_{k}J\ \EQUIV\ \skc{k}I=\skc{k}J}
(cf.\ \cite[VI, \S 2]{GJarS}). We assume that $X$ is fibrant (so
each mapping space \w{X(u,v)} is a cubical Kan complex).

Factor the structure map \w{p\ui{2}:\Po{2}X\to \Po{1}X} as a cofibration
\w{i\ui{2}:P_{2}X\to Y} followed by a weak equivalence, so that the
pushout above is in fact a homotopy pushout, as required.
Thus \w{Y_{i}=X_{i}} for \w[,]{i\leq 2} while
\w[,]{Y_{3}=(X_{3}\bsim_{2})\,\amalg\,\bar{Y}} where \w{\bar{Y}}
consists of a $3$-cube \w{J=J_{\Ic{}}} for every collection
\w{\Ic{}:=(I_{1}^{0},I_{1}^{1},\dotsc,I_{3}^{0},I_{3}^{1})} of $2$-cubes
in \w{X_{2}} with matching faces (such that
\w{d_{i}^{\varepsilon}J=I_{i}^{\varepsilon}} for \w{1\leq i\leq 3} and
\w[).]{\varepsilon\in\{0,1\}}
The pushout $Z$ is the reduction modulo \w{\sim_{2}} of
the image of \w{i\ui{2}} (without changing $\bar{Y}$). The
$3$-cubes \w{J_{(I,0,\dotsc,0)}} for non-null homotopic $I$
represent \w{\D=\hpi X} in \w[.]{\Po{3}Z\simeq\EM{\D}{\hpi X}{2}}

Let \w{W\simeq\N\D} be the cofibrant replacement for \w{\Po{1}X}
constructed as in \S \ref{scubtrack}. The weak equivalence
of \ww{\CO}-categories \w{\varphi:W\to Y} is then defined as follows:

\begin{enumerate}
\renewcommand{\labelenumi}{(\alph{enumi})\ }
\item For every indecomposable $0$-cube \w[,]{(f)\in\D_{0}=\F\ho\D}
corresponding to a homotopy class \w[,]{f\in[Xu,Xv]_{\ho\C}} \w{\varphi(f)}
is a choice of a representative \w{s(f))} in \w[.]{(P_{0}X)_{0}=X(u,v)_{0}}
\item For a (non-composite) $1$-cube $I$ corresponding to a track
\w{\xi:f\otimes g\Ra h} in \w[,]{W_{1}} the $1$-cube \w{\varphi(I)} is
a choice of a homotopy \w{s(\xi):s(f)\circ s(g)\Ra s(h)} representing $\xi$.
\item For a (non-composite) $2$-cube $J$ as in \wref[,]{eqsquare}
indexed by a diagram in $\D$ of the form \wref{eqtoda} satisfying
\wref[,]{eqthreesimp} the $2$-cube \w{\varphi(J)} is a choice of a
$2$-track \w{\alpha:f_{\ast}\eta\oplus\zeta\Ra h^{\ast}\theta\oplus\xi}
(cf.\ \S \ref{shtc}).
\item Finally, the faces of any $3$-cube $K$ form a set of matching
$2$-cubes, whose image under $\varphi$ has a canonical fill-in
\w[,]{T\in\bar{Y}} and we set \w[.]{\varphi(K):=T}
\end{enumerate}

The map \w{W\simeq \Po{1}X\to \Po{3}Z} represents \w[,]{k_{1}} and (as
in the proof of Theorem \ref{tbwso}) it is determined by its image on
the $3$-cubes of $W$. Using the description in \S \ref{scubtrack}(d),
we see that the $3$-cube:
$$
\xymatrix@R=25pt{
%
%
f_{1}\otimes f_{2}\otimes f_{3}\otimes f_{4}
\ar[rrrr]_>>>>>>>>>>>>>>>>>>>>>>{f_{1}\otimes f_{2}\otimes \xi^{[2,3,4]}}
\ar[rrd]^{f_{1}\otimes \xi^{[1,2,3]}\otimes f_{4}}
\ar[dd]^{\xi^{[0,1,2]}\otimes f_{3}\otimes f_{4}} &&&&
f_{1}\otimes f_{2}\otimes h^{[2,4]}
\ar[rrd]^{f_{1}\otimes \xi^{[1,2,4]}}
\ar[dd]^>>>>>{\xi^{[0,2]}\otimes h^{[2,4]}}
 && \\
%
%
&& f_{1}\otimes h^{[1,3]}\otimes f_{4}
\ar[rrrr]^<<<<<<<<<<<<{f_{1}\otimes \xi^{[1,3,4]}}
\ar[dd]^>>>>>>>{\xi^{[0,1,3]}\otimes f_{4}} &&&&
f_{1}\otimes h^{[1,4]} \ar[dd]_{\xi^{[0,1,4]}} \\
%
%
h^{[0,2]}\otimes f_{3}\otimes f_{4}
\ar[rrrr]^>>>>>>>>>>>>>>{h^{[0,2]}\otimes \xi^{[2,3,4]}}
\ar[rrd]_{\xi^{[0,2,3]}\otimes f_{4}} &&&&
h^{[0,2]}\otimes h^{[2,4]} \ar[rrd]^{\xi^{[0,2,4]}} && \\
%
%
&& h^{[0,3]}\otimes f_{4}
\ar[rrrr]^{\xi^{[0,3,4]}} &&&& h^{[0,4]}
}
$$
\noindent is sent under \w{k_{1}} to the cube in $\bar{Y}$ described
by replacing \w{f_{1}\otimes f_{2}\otimes f_{3}\otimes f_{4}} by
\w[,]{s(f_{1})\otimes s(f_{2})\otimes s(f_{3})\otimes s(f_{4})} and so on.
This is just what the cocycle of \wref{eqfourcocycle}
does to the corresponding $4$-simplex of \w[,]{\N\D} under the
isomorphism of Theorem \ref{tbwso}.
\end{proof}

%
%
\section*{Appendix A: \ Fibrancy conditions on double groupoids}
\label{afibgpd}

In this appendix we prove some technical facts about double groupoids:

%
%
\begin{propa}[\protect{\ref{pone}}]
Let \w{\Xd\in s\Gpd} be a simplicial groupoid, for which
the simplicial sets \w{X\bull_{0}} and \w{X\bull_{1}} are
\ww{\csk{2}}-fibrant, and
the morphism \w{d_{1}^{v}:X\bull_{1}\to X\bull_{0}} is a
\ww{\csk{2}}-fibration.
Then the left adjoint \w{\Ph:s\Gpd\to\TGpd} to the nerve
\w[,]{N^{h}:\TGpd\to s\Gpd} applied to \w[,]{\Xd} is \w[.]{\hpi^{h}\Xd}
\end{propa}

\begin{proof}
Let \w{X_{11}\bsim:=(\hpi^{h}\Xd)_{1}} and
\w{X_{10}\bsim:=(\hpi^{h}\Xd)_{0}} be the fundamental groupoids of the
\ww{\csk{2}}-fibrant (horizontal) simplicial sets \w{X\bull_{0}} and
\w[,]{X\bull_{1}} respectively, and let \w{[\alpha]\in X_{11}\bsim}
denote the equivalence class of \w[.]{\alpha\in X_{11}}

We first show that there is a well defined composition map:
\begin{myeq}\label{eqcomplaw}
\tens{X_{11}\bsim}{X_{10}\bsim}\supar{m} X_{11}\bsim
\end{myeq}
where \w{\ovl{d}^{v}_{i}:X_{11}\bsim\to X_{10}\bsim}
is induced by \w{ d^{v}_{i}:X_{11}\to X_{10}} for \w[.]{i=0,1}

For any \w[,]{([\alpha], [\beta])\in\tens{X_{11}\bsim}{X_{10}\bsim}}
we have \w{d^{v}_{1}([\alpha])=\ovl{d}^{v}_{0}([\beta])} \wh  that is,
$$
[\ovl{d}^{v}_{1}([\alpha])]=[\ovl{d}^{v}_{0}([\beta])]
$$
\noindent in \w[.]{X_{10}\bsim} This means that there is a \w{\tau\in X_{20}}
such that:
$$
d_{0}^{h}\tau~=~
d_{1}^{v}\alpha\hspace*{7mm}
d_{1}^{h}\tau=s_{0}^{h} d_{1}^{h} d_{1}^{v}\alpha
\hspace*{7mm}\text{and}\hspace*{5mm} d_{2}^{h}\tau=d_{0}^{v}\beta~.
$$
On the other hand, since \w{d_{1}^{v}:X\bull_{1}\to X\bull_{0}} is a
\ww{\csk{2}}-fibration, we have a lifting
$$
\xymatrix{
\Lambda^1(2)\ar^{(\alpha,s_{0}^{h} d_{1}^{h}\alpha)}[rr]
\ar[d] && X\bull_{1} \ar^{d_{1}^{v}}[d]\\
\Delta[2]\ar^{\xi}[rru] \ar_{\tau}[rr] && X\bull_{2}
}
$$
and \w[.]{d_{1}^{v}\xi=\tau}

The following picture summarizes the situation:
$$
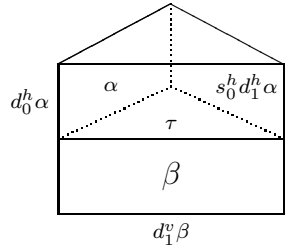
\xy
    (0,-2)*{}="1";
    (0,-13)*{}="8";
    (-15,-10)*{}="2";
    (15,-10)*{}="3";
    (-15,-20)*{}="4";
    (15,-20)*{}="5";
    (-15,-30)*{}="6";
    (15,-30)*{}="7";
    (-8,-13)*{\scriptstyle{\alpha}}="9";
    (10,-13)*{\scriptstyle{s_{0}^{h} d_{1}^{h}\alpha}}="10";
    (0,-18)*{\scriptstyle{\tau}}="11";
    (0,-25)*{{\beta}}="12";
    "1";"2"**\dir{-};
    "1";"3"**\dir{-};
    "2";"3"**\dir{-};
    {\ar@{-}_{d_{0}^{h}\alpha}"2";"4"};
    "1";"2"**\dir{-};
    "3";"5"**\dir{-};
    "4";"5"**\dir{-};
    "4";"5"**\dir{-};
    "4";"6"**\dir{-};
    "5";"7"**\dir{-};
    {\ar@{-}_{d_{1}^{v}\beta}"6";"7"};
    "4";"8"**\dir{.};
    "5";"8"**\dir{.};
    "1";"8"**\dir{.};
\endxy
$$

The prism at the top of the picture represents the filler \w{\xi\in X_{21}}
for the horn \w[.]{(\alpha,s_{0}^{h} d_{1}^{h}\alpha)} If we let
\w{\hat{\alpha}:=d_{2}^{h}\xi} (the front face of the prism), then we have:
$$
d_{0}^{h}\xi~=~\alpha,\hspace*{7mm} d_{1}^{h}=  s_{0}^{h} d_{1}^{h}\alpha,
\hspace*{7mm}\text{and}\hspace*{7mm} d_{2}^{h}\xi=\hat{\alpha}~.
$$
This means that \w{\alpha~\sim~\hat{\alpha}} in the equivalence
relation determined by the $2$-simplices of \w{X\bull_{1}} \wh that
is, \w{[\alpha]=[\hat{\alpha}]} in \w[.]{X_{11}\bsim}

Furthermore, since \w[,]{d_{1}^{v}\hat{\alpha} = d_{2}^{h}\tau=d_{0}^{v}\beta}
we see that \w{\hat{\alpha}} and \w{\beta} are composable in \w[.]{X_{1}}
We define the composition $m$ of \wref{eqcomplaw} by
\w[,]{m([\alpha],[\beta])= [\hat{\alpha}\circ\beta]}
where $\circ$ denotes the vertical composition in \w[.]{X_{1}}

To see that $m$ is independent of the choice of representatives for
\w{[\alpha]} and \w{[\beta]} and the lift $\xi$, suppose
\w[.]{\alpha\sim\alpha'} Then \w[,]{\hat{\alpha}\sim\hat{\alpha}'} so
there is a \w{\gamma\in X_{21}} with
$$
d_{0}^{h}\gamma=\hat{\alpha},\hspace*{7mm}
d_{1}^{h}\gamma=s_{0}^{h}d_{1}^{h}\hat{\alpha},
\hspace*{7mm}\text{and}\hspace*{7mm} d_{2}^{h}\gamma=\hat{\alpha}'
$$
Let \w[.]{\delta=\gamma \circ s_{0}^{h}\beta} Then
$$
d_{0}^{h}\delta=\hat{\alpha}\circ\beta~,
\hspace*{7mm} d_{1}^{h}\delta=(s_{0}^{h}d_{1}^{h}\hat{\alpha})\circ\beta,
\hspace*{7mm}\text{and}\hspace*{7mm} d_{2}^{h}\delta=\hat{\alpha}'\circ\beta~.
$$
This shows that \w{\hat{\alpha}\circ\beta\sim\hat{\alpha}'\circ\beta},
and therefore \w[.]{m([\alpha],[\beta])=m([\alpha'],[\beta'])}

Given \w[,]{\alpha\in X_{11}} consider two different choices of lifts \w{\xi}
yielding \w{\hat{\alpha}} and \w[.]{\hat{\alpha}'} Then
\w[.]{\hat{\alpha}\sim\hat{\alpha}'} Arguing as above, this implies that
\w[.]{\hat{\alpha}\circ\beta\sim\hat{\alpha}'\circ\beta}
We conclude that $m$ is well defined.

For all \w{([\alpha], [\beta], [\gamma])\in
\tens{X_{11}\bsim}{X_{10}\bsim}\tens{}{X_{10}\bsim} X_{11}\bsim} we have:
\begin{myeq}\label{eqassoc}
\begin{cases}~~~~
m([\alpha],m([\beta],[\gamma]))~=&~ m([\alpha],[\hat{\beta}\circ\gamma])=
[\hat{\alpha}\circ(\hat{\beta}\circ\gamma)]=
[(\hat{\alpha}\circ\hat{\beta})\circ\gamma]\\
~~~~m(m([\alpha],[\beta]),[\gamma])~=&~m([\hat{\alpha}\circ\beta],[\gamma])=
          [(\widehat{\hat{\alpha}\circ\beta})\circ\gamma]~,
\end{cases}
\end{myeq}
\noindent and since
\w[,]{\widehat{\hat{\alpha}\circ\beta}\sim\hat{\alpha}\circ\beta\sim
\hat{\alpha}\circ\hat{\beta}}
the two lines of \wref{eqassoc} are equal, so $m$ is associative.

Since the vertical elements in \w{X\bull_{1}} are
invertible with respect to $\circ$, which induces $m$,
we have a groupoid
$$
\xymatrix{
X_{11}\bsim\times_{X_{10}\bsim}X_{11}\bsim \ar[rr]^{m} &&
X_{11}\ar@<0.7ex>[rr]^{\ovl{d}^{v}_{0}}\ar@<-.7ex>[rr]_{\ovl{d}^{v}_{1}}
&&  X_{10}\bsim \ar@/_1.3pc/[ll]
}
$$
(in the notation of \wref[),]{eqgpoid} which we denote by
\w[.]{X_{1}\bull\bsim}

Notice that, from the above, there is an isomorphism for all \w{n\geq 2}:
$$
\underbrace{(X_{11}\times_{X_{10}}\times\dotsc
\times_{X_{10}}\times X_{11})}_{n}\bsim~
\cong~\underbrace{(X_{11}\bsim\times_{X_{10}\bsim}\times\dotsc
\times_{X_{10}\bsim~}\times X_{11}\bsim)}_{n}~.
$$

Now if \w{N^{h}:\TGpd\to s\Gpd} is the (horizontal) nerve functor,
from the above we have a double groupoid \w{\Ph X} with
$$
(N^{h}\Ph X)_{n}=\hpi (X\bull_{n}).
$$

The groupoid of objects and vertical morphisms of \w{\Ph X} is
\w[;]{X_{0}\bull} the groupoid of objects and horizontal morphisms
is \w[;]{\hpi (X\bull_{0})}
the groupoid of objects and horizontal morphisms and squares is
\w[;]{X_{1}\bull\bsim} and the groupoid of vertical morphism and squares is
\w[.]{\hpi^{h}(X\bull_{1})}

We now show that $\Ph$ is left adjoint to the nerve functor
\w[.]{N^{h}} Since \w{N^{h}} is fully faithful, a morphism \w{f:\Ph\Xd\to\Gss}
in \w{\TGpd} corresponds uniquely to a morphism
\w{N^{h}f:N^{h}\Ph\Xd\to N^{h}\Gss} in \w[.]{ s\Gpd} The latter amounts to
morphisms \w{(N^{h}f)_{n}:(N^{h}\Ph\Xd)_{n}\to(N^{h}\Gss)_{n}} in \w{\Gpd}
commuting with face and degeneracy operators.
But \w{(N^{h}\Ph\Xd)_{n}=N^{h}\hpi^{h}X\bull_{n})} and
\w[.]{(N^{h}\Gss)_{n}=N^{h}G_{\ast n}}
Thus each \w{(N^{h}f)_{n}} corresponds uniquely to a
morphism \w{\hpi^{h} X\bull_{n}\to G_{\ast n}} and thus, by adjunction, to a
morphism \w[.]{\ovl{f}_{n}:X\bull_{n}\to N^{h}G_{\ast n}} Using the fact
that \w[,]{\Ph N^{h}=\Id} it is straightforward to check that all the
\w{\ovl{f}_{n}} commute with face and degeneracy operators, hence
correspond uniquely to a morphism \w{\Xd\to N^{h}\Gss} in \w[.]{s\Gpd}
\end{proof}

%
%
\begin{propa}[\protect{\ref{ptwo}}]

Let \w{\Xdd\in s\Ss} be such that \w{X\bull_{i}} and \w{X_{i}\bull} are
\ww{\csk{2}}-fibrant for each \w[,]{i\geq 0} and
\w{d_{0}^{h}:X_{1}\bull\to X_{0}\bull} and
\w{d_{0}^{v}:X\bull_{1}\to X\bull_{0}} are \ww{\csk{2}}-fibrations. Then
\begin{enumerate}
\renewcommand{\labelenumi}{(\roman{enumi})\ }
\item \w{(N^{v}\hpi^{v}\Xdd)_{i}\bull} is fibrant for all \w[.]{i\geq 0}
\item \w{(N^{v}\hpi^{v}\Xdd)\bull_{1}} is \ww{\csk{2}}-fibrant.
\item \w{\ovl{d}^{v}_{0}:(N^{v}\hpi^{v}\Xdd)\bull_{1}\to
      (N^{v}\hpi^{v}\Xdd)\bull_{0}} is a \ww{\csk{2}}-fibration.
\item \w{\ovl{d}^{h}_{0}:(N^{v}\hpi^{v}\Xdd)_{1\bull}\to
        (N^{v}\hpi^{v}\Xdd)_{0}\bull} is a fibration.
\end{enumerate}
\end{propa}

\begin{proof}

\noindent (i) \ \w{(N^{h}\hpi^{v}\Xdd)_{i}\bull} is fibrant since it is the
nerve of a groupoid\vsm .

\noindent (ii) \ \ This follows as in the proof of Proposition A, with
vertical and horizontal directions switched\vsm .

\noindent (iii) \ \ Consider the diagram
$$
    \xymatrix {
\Delta[0] \ar^{\ovl{\alpha}}[r] \ar^{\alpha}[dr] \ar^{\cong}[d] &
        X\bull_{1} \ar^{q}[d] \ar@/^2pc/[dd]^{d_{0}^{v}}\\
\Lambda^k[1] \ar^{}[r] \ar@{^{(}->}[d] & X\bull_{1}\bsim
\ar^{\ovl{d}^{v}_{0}}[d]\\
\Delta[1]\ar^{}[r] & X\bull_{0}
    }
$$
Since \w{d_{0}^{v}} is a \ww{\csk{2}}-fibration, there is a lift
\w[.]{\xi:\Delta[1]\to X\bull_{1}} Then
\w{q\xi:\Delta[1]\to X\bull_{1}\bsim} is the required factorization.

It remains to show that there is a lift for
$$
    \xymatrix{
        \Lambda^k[2]\ar^{([\alpha],[\beta])}[rr]\ar@{^{(}->}_{}[d]&&
       X\bull_{1}\bsim \ar^{\ovl{d}_{0}^{h}}[d]\\
        \Delta[2]\ar_{}[rr] &&  X\bull_{0}
        }
$$
\noindent By (i) \ we know that \w{([\alpha],[\beta])} factors as
$$
\Lambda^k[2]\supar{(\hat{\alpha},\beta)}
      X\bull_{1}\supar{q}X\bull_{1}\bsim~.
$$
Since \w{d_{0}^{v}} is a \ww{\csk{2}}-fibration, it has a lift
\w[.]{\xi:\Delta[2]\to X\bull_{1}} Thus \w{q\xi} is the required
factorization\vsm.

\noindent (iv) \ \ By Lemma \ref{lfibr}, it
is enough to show that there is a lift in the diagram
\begin{myeq}\label{facts.eq7}
\begin{split}
    &   \xymatrix{
        \Lambda^k[1]\ar^{}[rr]\ar_{}[d]&&  X_{1}\bull\bsim \ar^{}[d]\\
        \Delta[1]\ar_{}[rr] &&  X_{0}\bull\bsim
        }\\
\end{split}
\end{myeq}
Notice that this factor as follows
$$
\xymatrix{
\Lambda^k[1]\ar^{}[rr]\ar_{}[d]&&  X_{1}\bull \ar^{d_{0}^{h}}[d]\ar^{q}[rr] &&
X_{1}\bull\bsim \ar^{\ovl{d}^{h}_{0}}[d]\\
        \Delta[1]\ar_{}[rr] &&  X_{0}\bull\ar^{}[rr] &&    X_{0}\bull\bsim
        }
$$
Since \w{d_{0}^{h}} is a \ww{\csk{2}}-fibration, there is a lift
\w[.]{\xi:\Delta[1]\to X_{1}\bull} Thus \w{q\xi} is the required lift in
\wref[.]{facts.eq7}
\end{proof}

%
%
\section*{Appendix B: \ $n$-diagonals of bisimplicial sets}
\label{andiag}

In this appendix we prove an elementary fact about bisimplicial sets,
which is presumably known, but which we have not found in the
literature. It follows immediately from the Bousfield-Friedlander
spectral sequence (cf.\ \cite[Theorem B.5]{BFrH}) in the cases where
it converges, but is actually true for any bisimplicial set.
In this paper we only need the result for \w[,]{n=2} in which case the
proof can be much simplified; but we prove the general case for future
reference.

%
%
\begin{propa}[\protect{\ref{pndiag}}]
If \w{f:\Wdd\to\Vdd} is an diagonal $n$-equivalence (Definition \ref{dndiag}),
then the induced map \w{\diag f:\diag\Wdd\to\diag\Vdd} is an $n$-equivalence.
\end{propa}

\begin{proof}
\textbf{Step I.}\hsm
First note that for any \w{\Wdd\in s\Ss} we have
\w{\diag\Wdd\simeq\hocolim_{k}~V_{k}^{v}} as a diagram in ``vertical''
simplicial sets (see \cite[XII, \S 3.4]{BKaH}), and since \wref{eqone}
is a Quillen equivalence, it suffices to prove the result for
$$
\hocolim_{k}|f^{v}|:\hocolim_{k}~|W_{k}^{v}|~\to~\hocolim_{k}~|V_{k}^{v}|~
$$
\noindent (See \cite[\S 18.9.8]{PHirM}).

For any \w[,]{\Vd\in s\C} we denote by \w{\tVd} the corresponding
\emph{restricted} simplicial object (also called \emph{ss}-object \wh
see \cite{KanS}), in which we forget the degeneracies: in other words,
\w[,]{\tVd\in\C^{\bDelta\res\op}} where \w{\bDelta\res} is the
category of finite ordered sets and order-preserving \emph{monomorphisms}.

By \cite[Prop.\ A.1(iv)]{SegCC}), the homotopy colimit of
the diagram \w{\{|W^{v}_{k}|:\bDelta\to\T\}_{k=0}^{\infty}} is weakly
equivalent to the homotopy colimit of
\w[,]{(\{|\tilde{W}^{v}|)_{k}:\bDelta\res\to\T\}_{k=0}^{\infty}}
which we denote by \w[.]{\tWd:\bDelta\res\to\T}

Finally, recall that \w{\hocolim \tXd} can be described explicitly as
%
\begin{myeq}\label{eqtwo}
\coprod_{n=0}^{\infty}~X_{n}\times\Delta[n]~/~\sim
\end{myeq}
\noindent where \w{(d_{i}x,\sigma)\sim(x,d^{i}\sigma)} for each
\w{x\in X_{n}} and \w[,]{\sigma\in\Delta[n-1]} where
\w{d^{i}:\Delta[n-1]\hra\Delta[n]} is the inclusion of the $i$-th face
(cf.\ \cite[\S 1]{SegC})\vsm.

\noindent\textbf{Step II.}\hsm
Note that, for any \w[,]{Y\in\Ss} \w{|Y|} is a CW complex,  and
realization commutes with skeleta: \w[.]{\sk{n}|Y|=|\sk{n}Y|}

For any restricted simplicial object \w{\tXd\in\C^{\bDelta\res\op}}
(where $\C$ is either $\Ss$ or $\T$), we define its
\emph{diagonal $n$-skeleton} to be the restricted simplicial object
\w{\dsk{n}\tXd} with \w{(\dsk{n}\tXd)_{k}=\sk{n-k}\tilde{X}_{k}}
(so \w{(\dsk{n}\tXd)_{k}=\emptyset} for \w[).]{k>n} Let
\w{j_{n}:\dsk{n}\tXd\hra\tXd} be the inclusion (a
diagonal \ww{(n-1)}-equivalence).

For any CW complex $W$, all cells in the relative CW complex
\w{(W\times\Delta[k],W\times\partial\Delta[k])} in dimensions
\w{\leq n} are products of standard $i$-simplices with cells of
\w[.]{\sk{n-k}W}
Thus we see from \wref{eqtwo} that when \w[,]{\C=\T} \w{j_{n}} induces
a cellular isomorphism of the $n$-skeleton of the homotopy colimit of
the diagram \w{\dsk{n}\tXd:\bDelta\res\to\T} with that of \w[,]{\tXd}
since they have the same cells in dimensions $\leq n$. In particular,
\w{\hocolim(j_{n}):\hocolim\,\dsk{n}\tXd\to\hocolim\,\tXd} is an
\ww{(n-1)}-equivalence. Thus if \w{\tXd=|\tZd|^{v}} for some
\w[,]{\tZd:\bDelta\res\to\Ss} then also
\w{\hocolim(j_{n}):\hocolim\,\dsk{n}\tZd\to\hocolim\,\tZd} is an
\ww{(n-1)}-equivalence\vsm.

\noindent\textbf{Step III.}\hsm
Recall that if \w{K\in\Ss} is a Kan complex, its $n$-th Postnikov
section \w{\Po{n}K} may be identified with the \ww{(n+1)}-st
coskeleton functor \w[,]{\csk{n+1}K} which factors through its left
adjoint, the \ww{(n+1)}-st skeleton (see \cite[\S 1.3]{DKanO})
Moreover, one can use any Postnikov tower functor on $\Ss$ or $\T$ to
obtain \emph{functorial} $k$-invariants, as in \cite[\S 5-6]{BDGoeR}.
What we shall in fact be doing in each case is to start with
\w{\csk{k+1}K} as the bottom section of our Postnikov tower, for the
appropriate $k$, and then use the pullbacks along the functorial
$k$-invariants to define the higher sections.  We may replace $K$ by
the limit of the resulting tower of fibrations.

Given a diagonal $n$-equivalence \w[,]{f:\Wdd\to\Vdd} let
\w{g:\tXd\to\tYd} denote the associated map
\w{\tilde{f}^{v}\bull:\tilde{W}^{v}\bull\to\tilde{V}^{v}\bull}
of restricted simplicial objects in $\Ss$.  We may assume that
\w{\tXd} and \w{\tYd} are Reedy fibrant, and $g$ is a Reedy fibration
(see \cite[\S 15]{PHirM}). In particular, for each \w[,]{n\geq 0}
\w{\tX_{n}=W^{v}_{n}} and \w{\tY_{n}=V^{v}_{n}} are Kan complexes, and
the maps \w{\delta_{n}^{X}:\tX_{n}\to M_{n}\tXd} and
\w{\delta_{n}^{Y}:\tY_{n}\to M_{n}\tYd} are fibrations, where the
\emph{matching object} of any (restricted) simplicial object \w{\tUd}
is defined
\begin{myeq}\label{eqmatch}
M_{n}\tUd=\{(u_{0},\dotsc,u_{n})\in(\tU_{n-1})^{n+1}~|\
d_{i}u_{j}=d_{j-1}u_{i} \text{\ \ for all\ }0\leq i<j\leq n\}
\end{myeq}
\noindent (see \cite[X, \S 4.5]{BKaH}). Note that all face maps
\w{d_{i}:\tU_{n}\to \tU_{n-1}} are defined by post-composing
\w{\delta_{n}^{U}:\tU_{n}\to M_{n}\tUd} with projection onto the
$i$-th factor \w[.]{\tU_{n-1}}

Finally, let \w[,]{g:=f\circ j_{n+1}:\tUd\to\tYd} where
\w{\tUd:=\dsk{n+1}\tXd} is the diagonal \ww{(n+1)}-skeleton\vsm .

\noindent\textbf{Step IV.}\hsm
We shall now construct a factorization
%
\begin{myeq}\label{eqthree}
\tUd~\xra{\psi}~\tZd~\xra{h}~\tYd
\end{myeq}
\noindent of $g$, with $h$ a vertical weak equivalence and
\w{\dsk{n+1}\psi} an isomorphism. We do so by induction on the
(horizontal) simplicial dimension \w[:]{m\geq 0}

\begin{enumerate}
\renewcommand{\labelenumi}{(\alph{enumi})\ }
\item Assume that we have defined \w{\tZd} and constructed the
  factorization \wref{eqthree} through (horizontal) simplicial
  dimension \w[,]{m-1} and let \w[.]{\ell:=n-m}

First, use \w{\Po{\ell}=\csk{\ell+1}} and the functorial $k$-invariants as
  above to define a commutative diagram in $\Ss$ as follows:
\mydiagram[\label{eqpost}]{
\tU_{m} \ar[rr]^{p^{(\ell)}} \ar[d]_{g_{m}} &&
\Po{\ell}\tU_{m} \ar[rr]^<<<<<<<<{(\Po{\ell}g_{m})^{\ast}k_{\ell}}
\ar[d]_{\simeq}^{\Po{\ell}g} & & K(\pi_{\ell+1}\tY_{m},\ell+2) \ar[d]_{=} \\
\tY_{m} \ar[rr]^{q^{(\ell)}} && \Po{\ell}\tY_{m} \ar[rr]^<<<<<<<<<{k_{\ell}} &&
K(\pi_{\ell+1}\tY_{m},\ell+2)
}
\noindent where \w{p^{(\ell)}} and \w{q^{(\ell)}} are the structure maps of
the Postnikov tower.  We define \w{\tZ_{m}\in\Ss} by means of its
Postnikov tower, starting with \w{\Po{\ell}(\tZ_{m}):=\Po{\ell}\tU_{m}}
with \w[.]{\Po{\ell}h_{m}=\Id} Let \w{\Po{\ell+1}(\tZ_{m})} be the
homotopy fiber of \w[.]{(\Po{\ell}g_{m})^{\ast}k_{\ell}}

Using the following (vertical) map of the two horizontal fibration
sequences on the right, we obtain a dotted map as indicated in:
\mydiagram[\label{eqpb}]{
\tU_{m} \ar[dd]^{g_{m}} \ar@{.>}[dr]_{\tilde{\psi}} \ar[drrrrr]^{p^{\ell}}
&&&&&\\
& \Po{\ell+1}\tZ_{m}\ar @{} [dr] |<<<<{\framebox{\scriptsize{PB}}}
 \ar[rr] \ar[d]_{\Po{\ell+1}h_{m}} &&
\Po{\ell}\tU_{m} \ar[rr]_{g^{\ast}k_{\ell}} \ar[d]_{\simeq}^{\Po{\ell}g_{m}} &&
K(\pi_{\ell+1}\tY_{m},\ell+2) \ar[d]_{=}\\
\tY_{m} \ar[r]^{q^{\ell+1}} &
\Po{\ell+1}\tY_{m} \ar[rr] && \Po{\ell}\tY_{m} \ar[rr]^{k_{\ell}} &&
K(\pi_{\ell+1}\tY_{m},\ell+2)
}
\noindent using the indicated pullback. This yields factorization:
%
\begin{myeq}\label{eqfour}
\Po{\ell+1}\tU_{m}~\xra{\Po{\ell+1}\psi_{m}}~
\Po{\ell+1}\tZ_{m}~\xra{\Po{\ell+1}h_{m}}~\Po{\ell+1}\tY_{m}
\end{myeq}
\noindent of \w[.]{\Po{\ell+1}g_{m}}

Since $g$ was a diagonal $n$-equivalence, \w{g_{m}} is an
$\ell$-equivalence, so \w{\Po{\ell}g_{m}} is a weak equivalence.
Thus \w{\Po{\ell+1}h_{m}} is a weak equivalence (since
we pulled back the $k$-invariants for \w{\tY_{m}} along a weak equivalence).

Note that \w{\Po{\ell}\psi_{m}} is an isomorphism of simplicial sets, so that
\w{\sk{\ell+1}\psi_{m}} is an isomorphism, too (since
\w[).]{\Po{\ell}=\csk{\ell+1}}
\item Proceeding as in (b) we pull back the rest of the Postnikov
system for \w{\tY_{m}} along weak equivalences, obtaining a weak
equivalence \w{h_{m}:\tZ_{m}\to\tY_{m}} and the factorization as in
\wref[.]{eqfour}
\item To complete the construction we must define the (horizontal)
  face maps of \w[,]{\tZd} where we assume by induction that they have
  been defined through (horizontal) simplicial dimension \w[.]{m-1} We
  also assume that the resulting \ww{(m-1)}-truncated restricted
  simplicial object in $\Ss$, \w[,]{\skh{m-1}\tZd} is Reedy fibrant, and
  that \w{\skh{m-1}h} is a Reedy fibration. Since \w{\tXd} (and thus
  \w[)]{\skh{m-1}\tXd} was already Reedy fibrant, and
  \w[,]{\skv{n-i}\tZ_{i}=\skv{n-i}\tU_{i}=\skv{n-i}\tX_{i}} making
  \w{\tZd} fibrant does not require any changes in (vertical)
  simplicial dimension \w[.]{\leq n-i} Similarly, when changing
  \w{\skh{m-1}h} into a fibration, we need make no changes in
  (vertical) simplicial dimension \w[.]{\leq n-i}

Note that since \w{\skh{m-1}\tZd} is Reedy fibrant, the limit defining
the matching object \w[,]{M_{m}\tZd} in \wref{eqmatch} is actually a
homotopy limit (by the dual of \cite[Cor.\ 19.5(2)]{CScheH}), and
similarly for \w[.]{M_{m}\tYd} Since \w{\skh{m-1}h} is a weak
equivalence, so is \w[.]{M_{m}h:M_{m}\tZd\to M_{m}\tYd} Moreover,
since limits preserve fibrations, \w{M_{m}h} is also a fibration. But
then we have a lifting:
\mydiagram[\label{eqlift}]{
\tU_{m} \ar[d]^{\text{cof}}_{\psi_{m}} \ar[rr]^{\skh{m}\delta^{U}_{m}}
&& M_{m}\tZd \ar[d]^{\simeq~\text{fib}}_{M_{m}h}\\
\tZ_{m} \ar@{.>}[rru]^{\delta^{Z}_{m}} \ar[rr]_{\delta^{Y}_{m}\circ h_{m}} &&
M_{m}\tYd~,
}
\noindent and \w{\delta_{m}^{Z}:\tZ_{m}\to M_{m}\tZd} determines the
(horizontal) face maps \w{d_{i}^{h}:\tZ_{m}\to\tZ_{m-1}} \wb[,]{0\leq i\leq m}
making
\w{\skh{m}\tZd} into an ($m$-truncated) restricted simplicial object,
\w{\skh{m}h:\skh{m}\tZd\to\skh{m}\tYd} into a weak equivalence, and
\w{\skh{m}\psi:\skh{m}\tUd\to\skh{m}\tZd} into a map between such
objects. This completes the induction\vsm.
\end{enumerate}

\noindent\textbf{Step V.}\hsm
To complete the proof, note that \w{\tZd} and \w{\tYd} are weakly
equivalent, so \w[.]{\hocolim~\tZd~\simeq\ \hocolim~\tYd} But by
construction \w{\dsk{n+1}\tZd=\tUd=\dsk{n+1}\tXd} (isomorphic
as restricted simplicial objects in $\Ss$), so by Step II,
\w{\hocolim~\tXd} and \w{\hocolim~\tZd} are $n$-equivalent. Thus by
Step I, \w{\diag\Wdd} and \w{\diag\Vdd} are $n$-equivalent, and a
diagram chase shows that \w{\diag f} is an $n$-equivalence, as claimed.
\end{proof}

%
%

%
\end{document}